\documentclass[onefignum,onetabnum]{siamart171218}

\usepackage{lipsum}
\usepackage{amsfonts}
\usepackage{graphicx}
\usepackage{epstopdf}
\usepackage{algorithmic}
\usepackage{graphicx,amsmath,amssymb,bm}
\ifpdf
  \DeclareGraphicsExtensions{.eps,.pdf,.png,.jpg}
\else
  \DeclareGraphicsExtensions{.eps}
\fi

\newsiamremark{remark}{Remark}
\newsiamremark{example}{Example}
\newsiamthm{assumption}{Assumption}

\newcommand{\abs}[1]{\left\vert#1\right\vert}
\newcommand{\I}[1]{\mathbb{I}\{#1\}}
\newcommand{\mrd}{\,\mathrm{d}}
\newcommand{\R}{\mathbb{R}}
\newcommand{\E}[1]{\mathbb{E}\left[#1\right]}

\headers{Error Rate of CQMC for Discontinuous Functions}{Z. He}

\title{On the Error Rate of Conditional Quasi-Monte Carlo for Discontinuous Functions\thanks{Submitted to the editors DATE: April, 2018.
\funding{This work was supported by the National Science Foundation of	China under grant 71601189.}}}

\author{Zhijian He\thanks{School of Mathematics, South China University of Technology, Guangzhou 510641, P. R. China
  (\email{hezhijian@scut.edu.cn}).}}

\usepackage{amsopn}

\ifpdf
\hypersetup{
  pdftitle={Error Rate of CQMC for Discontinuous Functions},
  pdfauthor={Z. He}
}
\fi

\begin{document}

\maketitle

\begin{abstract}
  This paper studies the rate of convergence for conditional quasi-Monte Carlo (QMC), which is a counterpart of conditional Monte Carlo. We focus on discontinuous integrands defined on the whole of $\R^d$, which can be unbounded. Under suitable conditions, we show that conditional QMC not only has the smoothing effect (up to infinitely times differentiable), but also can bring orders of magnitude reduction in integration error compared to plain QMC. Particularly, for some typical problems in options pricing and Greeks estimation, conditional randomized QMC that uses $n$ samples yields a mean error of $O(n^{-1+\epsilon})$ for arbitrarily small $\epsilon>0$. As a by-product, we find that this rate also applies to randomized QMC integration with all terms of the ANOVA decomposition of the discontinuous integrand, except the one of highest order.
\end{abstract}



\section{Introduction}
Conditional Monte Carlo (CMC) is widely used in stochastic simulation (see \cite{asmu:glyn:2007,fu:hu:1997}), which is also called conditioning. Suppose that our goal is to estimate an expectation (integral)
\begin{equation*}
I(f) = \E{f(\bm x)}= \int_{\R^d}f(\bm x)\rho_d(\bm x)\mrd \bm x,
\end{equation*}
where $d$ is the dimension of the problem, and $\rho_d(\bm x)$ is the probability density function (PDF) of the random vector $\bm x=(x_1,\dots,x_d)\in \R^d$. The basic idea of CMC is to use conditional
expectation of $f(\bm x)$ as an estimator. CMC enjoys a good effect of reducing the variance, compared to plain Monte Carlo (MC).
On the other hand, thanks to another effect of smoothing, CMC is widely used in sensitivity estimation when the problem involves discontinuities (see \cite{fu:2009,fu:hu:1997}). In practice, there are two major concerns in using CMC:
\begin{itemize}
	\item the choice of conditioning variables (say, $\bm z$), and
	\item the tractability of the resulting conditional expectation $\E{f(\bm x)|\bm z}$.
\end{itemize}
In this paper, we restrict our attention to the case of choosing some components of $\bm x$ as the conditioning variables $\bm z$ and assume that the components of $\bm x$ are independent identically distributed (IID). We focus on investigating the smoothness property of the resulting conditional expectation $\E{f(\bm x)|\bm z}$ rather than inspecting the tractability of $\E{f(\bm x)|\bm z}$.

Quasi-Monte Carlo (QMC) is usually applied to integration problems over the unit cube, which yields an asymptotic error rate of $O(n^{-1}(\log n)^d)$ when the integrand has bounded variation in the sense of Hardy and Krause \cite{nied:1992}.
Conditional QMC (CQMC) is a counterpart of CMC by replacing the random points with QMC points. We consider a
setting that the integrand $f(\bm x)$ is discontinuous, under which plain QMC may lose its power because QMC favors smooth integrands. He and Wang \cite{he:wang:2015} and He \cite{he:2017} gave convergence rates of randomized QMC (RQMC) for certain classes of discontinuous functions. The rates decline quickly as the dimension $d$ goes up. CQMC has the potential to improve the efficiency of QMC as conditioning could smooth the integrand more or less. Our main interest is to provide theoretical guarantees for using CQMC. 

A necessary first step in applying QMC methods to an integral over $\R^d$ is to transform the integral into an integral over the unit cube $(0,1)^d$. That transformation may introduce singularities at the boundary of $(0,1)^d$. In general, conditioning cannot remove such singularities, but it brings a smoothing effect \cite{grie:2010,grie:2013,grie:2016}. Griebel et al. \cite{grie:2013,grie:2016} studied kink functions of the form $f(\bm x)=\max(\phi(\bm x),0)$, where $\phi$ is a smooth function on $\R^d$, and showed that under suitable conditions, integrating out some components of $\bm x$ (this process is actually the conditioning method in our terminology) leads to a function with unlimited smoothness. Griebel et al. \cite{grie:2010} considered the setting of integration problems over the domain $[0,1]^d$. More recently, Griewank et al. \cite{grie:2017} considered a smoothing method called	``preintegration". In the preintegration method, one of the variables is integrated out for non-smooth integrands with kinks or discontinuities. By extending the work in \cite{grie:2013,grie:2016}, Griewank et al. \cite{grie:2017} proved that the presmoothed integrand belongs to an appropriate mixed derivative function space. However, these papers do not give error analysis for the smoothed function. Particularly, Griebel et al. \cite{grie:2013}  commented that 

``These results are
expected to lay the foundation for a future rigorous error analysis of direct numerical
methods for option pricing integrals over $\R^d$, methods that do not involve mapping
$\R^d$ to the unit cube."

Motivated by a sequence of papers by Griebel et al. \cite{grie:2010,grie:2013,grie:2016} and Griewank et al. \cite{grie:2017}, we first study the smoothness property of conditioning for certain discontinuous functions, which often arise in the pricing and hedging of financial derivatives. We then give conditions such that the resulting function satisfies the so-called boundary growth condition studied in Owen \cite{owen:2006}. The error analysis for CQMC is thus carried out by applying the results in Owen \cite{owen:2006}. In particular, we show that  conditional RQMC yields a mean error rate of $O(n^{-1+\epsilon})$ for arbitrarily small $\epsilon>0$ under some conditions. As illustrative examples, we show that the rate $O(n^{-1+\epsilon})$ is attainable for arithmetic Asian options with their Greeks and binary options, when using proper  constructions of the Brownian motion and conditioning variables. It is known that using dimension reduction methods in QMC can enhance the efficiency of QMC \cite{imai:tan:2006,wang:sloan:2011}. The rate $O(n^{-1+\epsilon})$ also holds if one uses some dimension reduction methods to combine with CQMC.  As a by-product, we give error rates for RQMC integration with all terms of the ANOVA decomposition of the discontinuous integrand. Under some conditions, RQMC can achieve a mean error rate of $O(n^{-1+\epsilon})$ for all ANOVA terms, except the one of highest order. While for the highest order term which is non-smooth, the rate may be just $O(n^{-1/2-1/(4d-2)+\epsilon})$, as found in He \cite{he:2017}.

To summarize, we make the following contributions in this paper.
\begin{itemize}
	\item We extend the work of \cite{grie:2013,grie:2016} to  discontinuous integrands by studying the smoothness property of conditioning. The analysis in  \cite{grie:2013,grie:2016} is based on the framework of  Sobolev space, while our analysis relies on uniform convergence conditions for some improper integrals.
	\item More importantly, we give rates of convergence for CQMC. Our analysis does not rely on the concrete form of the resulting conditional expectation. Additionally, the required conditions are very easy to check for our applications.
\end{itemize}
We find that the choice of conditioning variables is very important in CQMC as it has an impact on the smoothness of the resulting estimate and hence on the QMC accuracy. The theoretical underpinnings  in this paper are expected to predict the benefits of using CQMC in real-world applications.

The remainder of the paper is organized as follows. We formulate the problem in Section~\ref{sec:formulation}. Section~\ref{sec:smooth} studies the smoothing effect of conditioning. Section~\ref{sec:main} establishes rigorous error analysis for CQMC sampling. Section~\ref{sec:anova} studies the error rate of QMC integration with all terms of ANOVA decomposition. Several  examples from financial engineering are studied in Section~\ref{sec:ex} to exemplify the value of
our theoretical underpinnings, followed by concluding remarks in Section~\ref{sec:concl}.

\section{Problem Formulation}\label{sec:formulation}
Consider an expectation (or equivalently an integral over $\R^d$)
\begin{equation}\label{eq:integral}
I(f) = \E{f(\bm x)}= \int_{\R^d}f(\bm x)\prod_{i=1}^d\rho(x_i)\mrd \bm x,
\end{equation}
where the components of $\bm x$ are IID with PDF $\rho$ and cumulative distribution function (CDF) $\Phi$. For simplicity, we use the same notation for random variable and the integration variable. Throughout this paper, assume that $\E{\abs{f(\bm x)}}<\infty$.
To estimate the integral \eqref{eq:integral} by QMC, one may transform \eqref{eq:integral} into an integral over $(0,1)^d$
\begin{equation*}
I(f)=\int_{(0,1)^d} f(\Phi^{-1}(\bm u))\mrd \bm u,
\end{equation*}
where the inverse function $\Phi^{-1}$ applies to each component of $\bm u$.
We then take the following quadrature rule as an estimate of $I(f)$,
\begin{equation}\label{eq:qmc}
\hat{I}(f)=\frac{1}{n}\sum_{i=1}^{n}f(\Phi^{-1}(\bm u_i)),
\end{equation}
where $\bm u_i\in (0,1)^{d}$. In this paper, we are interested in discontinuous integrands over $\R^d$ of the form
\begin{equation}\label{eq:integrand}
f(\bm x)= g(\bm x)\I{\phi(\bm x)\geq 0},
\end{equation}
where $g, \phi$ are smooth functions of all variables. See Section~\ref{sec:ex} for examples of this form.

Denote $\bm x_{-j}$ as the $d-1$ components of $\bm x$ apart from $x_j$.
Integrating \eqref{eq:integrand} with respect to $x_j$ (i.e., taking $\bm x_{-j}$ as the conditioning variables) gives
\begin{equation*}\label{eq:condint}
(P_j f)(\bm x_{-j}):=\E{f(\bm x)|\bm x_{-j}}=\int_{-\infty}^{\infty} f(x_j,\bm x_{-j})\rho(x_j)\mrd x_j.
\end{equation*}
We should note that $I(f)=I(P_j f)$. RQMC integration with  $P_j f$ renders unbiased estimate, as the CMC sampling; see \cite{lecu:lemi:2005} for a survey on RQMC.  

Prior to studying the smoothness property of the function $P_j f$, we specify some notations. Denote $1{:}d=\{1,2,\dots,d\}$ and
$D_j\phi:=\partial \phi/\partial x_j$. For $v\subseteq 1{:}d$, $D_v \phi$ denotes the derivative taken with respect to each $x_j$ once  for all $j\in v$.  For any multi-index $\bm \alpha=(\alpha_1,\dots,\alpha_d)$ whose components are nonnegative integers, 
\begin{equation*}
(D^{\bm \alpha} \phi)(\bm x):=\frac{\partial^{\abs{\bm \alpha}}\phi}{\partial x_1^{\alpha_1}\dots x_d^{\alpha_d}}(\bm x),
\end{equation*}
where $\abs{\bm\alpha}=\sum_{i=1}^{d}\alpha_i$. If $\alpha_i=1$ for all $i\in v$ and $\alpha_i=0$ otherwise, then $D^{\bm \alpha} \phi=D_v \phi$.

\section{The Smoothing Effect of Conditioning}\label{sec:smooth}
In this section, we study the smoothness property of $P_j f$. A key condition we require below is the uniform convergence for improper integrals with parameters. That condition ensures the interchange of differentiation and integration.

\begin{definition}
	Let $\bm x\in \mathbb{R}^s$. An  integral $\int_{-\infty}^\infty f(t,\bm x)\mrd t$ converges uniformly on a set $\Theta\subseteq \mathbb{R}^s$ if for any $\epsilon>0$, there exists a constant $A_0>0$ depending on $\epsilon$ such that 
	$$\abs{\int_{\abs{t}>A}f(t,\bm x)\mrd t}<\epsilon$$
	for all $A>A_0$ and all $\bm x\in \Theta$.
\end{definition}

\begin{theorem}\label{thm:changes}
	Let $\Omega$ be an open set of $\mathbb{R}^s$, and let $f(t,\bm x)$ be a function defined over $\mathbb{R}\times \Omega$. Suppose that 
	\begin{itemize}
		\item $f(t,\bm x)$  and $\partial f/\partial x_i$ are continuous functions over $\mathbb{R}\times \Omega$, where $i\in 1{:}s$;
		\item $\int_{-\infty}^\infty f(t,\bm x)\mrd t$ exists for any $\bm x\in\Omega$; and 
		\item for any $\bm x^*\in\Omega$, there exists a set $B(\bm x^*,\delta)\subseteq \Omega$ with $\delta>0$ such that the integral $\int_{-\infty}^\infty \frac{\partial }{\partial x_i}f(t,\bm x)\mrd t$ converges uniformly on the set $B(\bm x^*,\delta)$.
	\end{itemize}
	Then 
	\begin{equation}\label{eq:order}
	\frac{\partial }{\partial x_i}\int_{-\infty}^\infty f(t,\bm x)\mrd t=\int_{-\infty}^\infty \frac{\partial }{\partial x_i}f(t,\bm x)\mrd t,
	\end{equation}
	which is continuous on $\Omega$. If $\psi(\bm x)\in \mathcal{C}^1(\mathbb{R}^s)$, then
	\begin{align}
	\frac{\partial }{\partial x_i}\int_{-\infty}^{\psi(\bm x)} f(t,\bm x)\mrd t&=\int_{-\infty}^{\psi(\bm x)} \frac{\partial }{\partial x_i} f(t,\bm x)\mrd t + f(\psi(\bm x),\bm x)\frac{\partial }{\partial x_i}\psi(\bm x),\label{eq:change1}\\\frac{\partial }{\partial x_i}\int_{\psi(\bm x)}^{\infty}f(t,\bm x)\mrd t&=\int_{\psi(\bm x)}^{\infty} \frac{\partial }{\partial x_i} f(t,\bm x)\mrd t - f(\psi(\bm x),\bm x)\frac{\partial }{\partial x_i}\psi(\bm x),\label{eq:change2}
	\end{align}
	which are both continuous on $\Omega$.
\end{theorem}
\begin{proof}
	See \cite{buck:2003} for the proof of interchanging the order of differentiation and integration in \eqref{eq:order}.  Equations \eqref{eq:change1} and \eqref{eq:change2} are consequences of applying the classic Leibniz rule for improper integrals. 
\end{proof}

\begin{assumption}\label{assump1}
	Let $j\in 1{:}d$ be fixed. Assume that 
	\begin{equation*}\label{eq:assump1}
	(D_j\phi)(\bm x)\neq 0 \text{ for all } \bm x\in \R^d.
	\end{equation*}
\end{assumption}

\begin{theorem}[Implicit Function Theorem]\label{thm:IFT}Let $r$ be a positive integer. Denote $U_j = \{\bm x_{-j}\in \R^{d-1}|\phi(x_j,\bm x_{-j})=0 \text{ for some }x_j\in \R\}$. If $\phi\in \mathcal{C}^r(\R^d)$ and Assumption~\ref{assump1} is satisfied, then $U_j$ is open, and there exists a unique function $\psi\in \mathcal{C}^r(U_j)$ such that
	\begin{equation*}
	\phi(\psi(\bm x_{-j}),\bm x_{-j}) = 0\text{ for all }\bm x_{-j}\in U_j,
	\end{equation*}
	and for all $k\neq j$, we have
	\begin{equation*}\label{eq:dpsi}
	(D_k \psi)(\bm x_{-j}) = -\frac{(D_k\phi)(\bm x)}{(D_j\phi)(\bm x)}\bigg|_{x_j=\psi(\bm x_{-j})}
	\end{equation*}
	for all $\bm x_{-j}\in U_j$.
\end{theorem}
\begin{proof}
	See the proof of  Theorem 2.3 in \cite{grie:2013}. 
\end{proof}

\begin{theorem}\label{thm:grie}
	Let $r$ be a positive integer. Suppose that $f$ is given by \eqref{eq:integrand} with $g,\phi\in  \mathcal{C}^r(\R^d)$ and $\E{\abs{f(\bm x)}}<\infty$, $\rho\in \mathcal{C}^{r-1}(\R)$, and Assumption~\ref{assump1} is satisfied. Denote $\bm y=\bm x_{-j}$. Let
	\begin{align*}
	U_j &= \{\bm y\in \R^{d-1}|\phi(x_j,\bm y)=0 \text{ for some }x_j\in \R\},\label{eq:uj}\\
	U_j^+ &= \{\bm y\in \R^{d-1}|\phi(x_j,\bm y)>0 \text{ for all  }x_j\in \R\},\notag\\
	U_j^- &= \{\bm y\in \R^{d-1}|\phi(x_j,\bm y)<0 \text{ for all  }x_j\in \R\}.\notag
	\end{align*}
	Then $U_j$ is open, and there exists a unique function $\psi\in \mathcal{C}^r(U_j)$ such that 
	$\phi(\psi(\bm y), \bm y) = 0$ for all $\bm y\in U_j$. 
	Assume that for any $\bm y^*\in U_j$, there exists a set $B(\bm y^*,\delta)\subseteq U_j$ with $\delta>0$ such that  $\int_{-\infty}^\infty D^{\bm \alpha} g(
	x_j,\bm y)\rho(x_j)\mrd x_j$ converges uniformly on $B(\bm y^*,\delta)$ for any multi-index $\bm{\alpha}$ satisfying $\abs{\bm{\alpha}}\le r$ and $\alpha_j=0$.
	Assume also that every function over $U_j$ of the form
	\begin{equation}\label{eq:hy}
	h(\bm y) = \beta\frac{(D^{\bm \alpha^{(0)}}g)(\psi(\bm y),\bm y)\prod_{i=1}^a(D^{\bm \alpha^{(i)}}\phi)(\psi(\bm y),\bm y)}{[(D_j \phi)(\psi(\bm y),\bm y)]^b}\rho^{(c)}(\psi(\bm y)),
	\end{equation}
	where $\beta$ is a constant, $a,b,c$ are integers, and $\bm \alpha^{(i)}$ are multi-indices with the constraints $1\le a\leq 2r-1$, $1\le b\le 2r-1$, $0\le c\le r-1$, $\abs{\bm\alpha^{(i)}}\le r$, satisfies 
	\begin{equation}\label{eq:bounds}
	h(\bm y)\to 0\text{ as }\bm y\text{ approaches a boundary point of }U_j\text{ lying in }U_j^+\text{ or }U_j^-.
	\end{equation}
	Then $P_j f\in \mathcal{C}^r(\R^{d-1})$, and for every multi-index $\bm \alpha$ with $\abs{\bm \alpha}\leq r$ and $\alpha_j=0$,
	\begin{equation}\label{eq:bound4dfj}
	\abs{(D^{\bm \alpha} P_j f)(\bm y)}\le\int_{-\infty}^{\infty}\abs{(D^{\bm \alpha} g)(x_j,\bm y)}\rho(x_j)\mrd x_j+\sum_{i=1}^{M_{\abs{\bm \alpha}}} \abs{h_{\bm \alpha,i}(\bm y)},
	\end{equation}
	where $M_{\abs{\bm \alpha}}$ is a nonnegative integer depending only on $\abs{\bm \alpha}$, and for $\bm y\in U_j$ and $\abs{\bm{\alpha}}>0$, $h_{\bm \alpha,i}(\bm y)$ has the form \eqref{eq:hy} with parameters satisfying $1\le a\leq 2\abs{\bm \alpha}-1$, $1\le b\le 2\abs{\bm \alpha}-1$, $0\le c\le \abs{\bm \alpha}-1$, $|\bm\alpha^{(i)}|\le \abs{\bm \alpha}$, otherwise $h_{\bm \alpha,i}(\bm y)=0$.
\end{theorem}
\begin{proof}
	This proof benefits largely from the proof of Theorem 1 in \cite{grie:2016}. The implicit function theorem guarantees the existence of the solution of $\phi(x_j,\bm y)=0$ for any $\bm  y\in U_j$. Without loss of generality, we suppose that $(D_j\phi)(\bm x)>0$ in Assumption~\ref{assump1}. This implies that $\phi(x_j,\bm y)$ is an increasing function with respect to $x_j$ for given $\bm y$. We then have
	\begin{equation*}
	\{\bm x|\phi(\bm x)\geq 0\}=\{\bm x|x_j\geq \psi(\bm y) \text{ for all }\bm y\in U_j\},
	\end{equation*}
	where $\psi\in \mathcal{C}^r(U_j)$.
	So the function $P_j f$ can be rewritten as
	\begin{equation*}\label{eq:smoothfn}
	(P_j f)(\bm y)=\displaystyle\begin{cases}
	\displaystyle\int_{-\infty}^{\infty}g(x_j,\bm y)\rho(x_j)\mrd x_j,&\ \bm y\in U_j^+\\
	\displaystyle\int_{\psi(\bm y)}^{\infty}g(x_j,\bm y)\rho(x_j)\mrd x_j,&\ \bm y\in U_j\\
	0,&\ \bm y\in U_j^-.
	\end{cases}
	\end{equation*}
	
	Let us consider the partial derivative of $P_j f(\bm y)$ for $\bm y\in U_j$. For $k\neq j$, applying the  Leibniz  rule \eqref{eq:change2} gives
	\begin{equation}\label{eq:dk}
	(D_k P_j f)(\bm y)=\int_{\psi(\bm y)}^{\infty}(D_k g)(x_j,\bm y)\rho(x_j)\mrd x_j-g(\psi(\bm y),\bm y)\rho(\psi(\bm y))(D_k \psi)(\bm y),
	\end{equation}
	which is continuous on $U_j$.
	The implicit function theorem admits
	\begin{equation*}
	(D_k \psi)(\bm y)=-\frac{(D_k \phi)(\psi(\bm y),\bm y)}{(D_j \phi)(\psi(\bm y),\bm y)}.
	\end{equation*}
	Thus \eqref{eq:dk} turns out to be
	\begin{equation*}
	(D_k P_j f)(\bm y)=\int_{\psi(\bm y)}^{\infty}(D_k g)(x_j,\bm y)\rho(x_j)\mrd x_j+g(\psi(\bm y),\bm y)\rho(\psi(\bm y))\frac{(D_k \phi)(\psi(\bm y),\bm y)}{(D_j \phi)(\psi(\bm y),\bm y)}.
	\end{equation*}
	Similarly, for $\ell\neq j$, we have
	\begin{equation*}\label{eq:dl}
	(D_\ell D_k P_j f)(\bm y)=\int_{\psi(\bm y)}^{\infty}(D_\ell D_k g)(x_j,\bm y)\rho(x_j)\mrd x_j+A(\psi(\bm y),\bm y),
	\end{equation*}
	where 
	\begin{align*}
	A(x_j,\bm y)&=(D_kg)(x_j,\bm y)\rho(x_j)\frac{(D_\ell \phi)(x_j,\bm y)}{(D_j \phi)(x_j,\bm y)}\\
	&+(D_\ell g)(x_j,\bm y)\rho(x_j)\frac{(D_k \phi)(x_j,\bm y)}{(D_j \phi)(x_j,\bm y)}\\
	&-(D_j g)(x_j,\bm y)\rho(x_j)\frac{(D_k \phi)(x_j,\bm y)(D_\ell \phi)(x_j,\bm y)}{[(D_j \phi)(x_j,\bm y)]^2}\\
	&-g(x_j,\bm y)p'(x_j)\frac{(D_k \phi)(x_j,\bm y)(D_\ell \phi)(x_j,\bm y)}{[(D_j \phi)(x_j,\bm y)]^2}\\
	&+g(x_j,\bm y)\rho(x_j)\frac{(D_\ell D_k \phi)(x_j,\bm y)}{(D_j \phi)(x_j,\bm y)}\\
	&-g(x_j,\bm y)\rho(x_j)\frac{(D_j D_k \phi)(x_j,\bm y)(D_\ell \phi)(x_j,\bm y)}{[(D_j \phi)(x_j,\bm y)]^2}\\
	&-g(x_j,\bm y)\rho(x_j)\frac{(D_k \phi)(x_j,\bm y)(D_\ell D_j \phi)(x_j,\bm y)}{[(D_j \phi)(x_j,\bm y)]^2}\\
	&+g(x_j,\bm y)\rho(x_j)\frac{(D_k \phi)(x_j,\bm y)(D_\ell \phi)(x_j,\bm y)(D_j D_j \phi)(x_j,\bm y)}{[(D_j \phi)(x_j,\bm y)]^3}.
	\end{align*}
	In general, for every multi-index $\bm \alpha$ with $\abs{\bm \alpha}\leq r$ and $\alpha_j=0$, one can conclude by induction on $\abs{\bm \alpha}$ that
	\begin{equation*}\label{eq:du}
	(D^{\bm \alpha} P_j f)(\bm y)=\int_{\psi(\bm y)}^{\infty}(D^{\bm \alpha} g)(x_j,\bm y)\rho(x_j)\mrd x_j+\sum_{i=1}^{M_{\abs{\bm \alpha}}} h_{\bm \alpha,i}(\bm y),
	\end{equation*}
	where $M_{\abs{\bm \alpha}}$ is a nonnegative integer depending on $\abs{\bm \alpha}$, and each function $h_{\bm \alpha,i}$ has the form \eqref{eq:hy} 
	with integers $\beta,a,b,c$   and multi-indices $\bm \alpha^{(i)}$  satisfying $1\le a\leq 2\abs{\bm \alpha}-1$, $1\le b\le 2\abs{\bm \alpha}-1$, $0\le c\le \abs{\bm \alpha}-1$, $\abs{\bm\alpha^{(i)}}\le \abs{\bm \alpha}$. Also,  $(D^{\bm \alpha} P_j f)(\bm y)$ is continuous on $U_j$.
	
	For $\bm y\in \mathrm{interior}(U_j^+)$,  applying Theorem~\ref{thm:changes} gives 
	$$(D^{\bm \alpha} P_j f)(\bm y)=\int_{-\infty}^{\infty}(D^{\bm \alpha} g)(x_j,\bm y)\rho(x_j)\mrd x_j,$$ 
	which is continuous on 	$\mathrm{interior}(U_j^+)$. 
	For $\bm y\in \mathrm{interior}(U_j^-)$, we have $(D^{\bm \alpha} P_j f)(\bm y)=0$. 
	
	Note that  $$U_j^-=\left\lbrace\bm y\in \R^{d-1}\bigg|\lim_{x_j\to \infty}\phi(x_j,\bm y)\le 0\right\rbrace.$$
	As a result, $\psi(\bm y)\to \infty$ as $\bm y$ approaches a boundary point of $U_j$ lying in $U_j^-$. Also,
	$$U_j^+=\left\lbrace\bm y\in \R^{d-1}\bigg|\lim_{x_j\to -\infty}\phi(x_j,\bm y)\ge 0\right\rbrace.$$
	Similarly, $\psi(\bm y)\to -\infty$ as $\bm y$ approaches a boundary point of $U_j$ lying in $U_j^+$. It then follows the condition~\eqref{eq:bounds}  that $D^{\bm \alpha} P_j f$ is continuous across the boundaries between $U_j$, $U_j^+$ and $U_j^-$. As a result, $P_j f\in \mathcal{C}^{r}(\R^{d-1})$ and the inequality \eqref{eq:bound4dfj} holds immediately. 
\end{proof}

When $U_j=\varnothing$, it reduces to the smooth scheme, i.e., $f(\bm x)=g(\bm x)$ or $f(\bm x)=0$. When $U_j=\R^{d-1}$, $U_j^+=U_j^-=\varnothing$. For the two extreme cases, the condition \eqref{eq:bounds} is satisfied automatically.   
We should note that Assumption~\ref{assump1} is critical to ensure a good smoothing effect of conditioning. As we will see in Section~\ref{sec:ex}, if Assumption~\ref{assump1} is violated, $P_j f$ is just continuous, but not differentiable. The uniform convergence is also critical in establishing Theorem~\ref{thm:grie}. 
The simplest standard test of the uniform convergence of an improper integral with parameters is the Weierstrass  test (see, e.g., \cite{buck:2003}), which will be used for the CQMC error analysis.

\begin{theorem}[Weierstrass  Test]
	Let $f(t,\bm x)$ be a function defined on $\R\times \Theta$. If there exists a function $F(t)$ defined on $\R$ such that $\sup_{\bm x\in \Theta}\abs{f(t,\bm x)}\le F(t)$ for all $t\in \R$ and $\int_{-\infty}^{\infty}F(t)\mrd t<\infty$, then $\int_{-\infty}^{\infty}f(t,\bm x)\mrd t$ converges uniformly on $\Theta$.
\end{theorem}
\begin{proof}
	Since $F(t)$ is integrable, for any $\epsilon>0$, there exists a constant $A_0$ such that $$\int_{\abs{t}>A} F(t) \mrd t<\epsilon$$ holds for any $A>A_0$. As a result,
	\begin{equation*}
	\abs{\int_{\abs{t}>A}f(t,\bm x)\mrd t}\le \int_{\abs{t}>A} F(t) \mrd t<\epsilon.
	\end{equation*}
	The uniform convergence immediately stands because $A_0$ depends on $\epsilon$ but not $\bm x$.
\end{proof}

Griewank et al. \cite{grie:2017} considered the isotropic Sobolev space with weight functions which  generalizes the setting in \cite{grie:2013}. In the following error analysis, we only require the existence of the mixed derivatives of $P_j f$ up to order $d-1$ and hence the uniform convergence conditions in Theorem~\ref{thm:grie} are sufficient.

\section{Error Analysis for CQMC}\label{sec:main}
Under the transformation $\bm x = \Phi^{-1}(\bm u)$, $f(\bm x)$ given by \eqref{eq:integrand} is then changed to
\begin{equation*}
q(\bm u):=f( \Phi^{-1}(\bm u)) = g( \Phi^{-1}(\bm u))\I{\bm u\in \Omega},
\end{equation*}
where 
\begin{equation}\label{eq:omega}
\Omega =  \{\bm u|\phi(\Phi^{-1}(\bm u))\ge 0\}\subseteq [0,1]^d.
\end{equation}
The function $q$ may have singularities along boundary of the unit cube $[0,1]^{d}$. He \cite{he:2017} showed that under certain conditions, RQMC integration with the function $q$ yields a mean error of $O(n^{-1/2-1/(4d-2)+\epsilon})$ for arbitrarily small $\epsilon>0$. The rate for discontinuous functions deteriorates quickly as the dimension $d$ goes up. As we will see, smoothing the integrand is a promising way to improve QMC accuracy.

Let $q_j(\bm u_{-j}):=P_j f(\Phi^{-1}(\bm u_{-j}))$. Although conditioning leads to a smooth effect, the smoothed function $q_j$ may also have singularities along boundary of the unit cube $[0,1]^{d-1}$. 
For functions satisfying the boundary growth condition (defined below), Owen \cite{owen:2006} found a mean error rate for RQMC and a similar error rate for Halton sequence. We are going to give conditions on $f$ and $\rho$ such that the function $q_j$ satisfies the boundary growth condition. The convergence rates in Owen \cite{owen:2006} can therefore be applied to the CQMC estimate.  

In this paper, we focus on RQMC integration using  scrambled $(t,s)$-sequences in base $b\ge 2$  proposed by Owen \cite{owen:1995} as inputs.  Here we do not restrict that $s=d$ because sometimes $s<d$ refers to the dimension of the CQMC estimate. In what follows, we assume that the points $\bm u_1,\dots,\bm u_n$ in the quadrature rule $\hat{I}$ defined by \eqref{eq:qmc} are the first $n$ points of a scrambled $(t,s)$-sequence in base $b\ge 2$. The error analysis for deterministic QMC integration with Halton sequence is similar.

\begin{definition}
	A function $g(\bm u)$ defined on $(0,1)^s$ is said to satisfy the boundary growth condition if
	\begin{equation}\label{eq:grow}
	\abs{D_{v}g(\bm u)}\leq B\prod_{i\in v}\min(u_i,1-u_i)^{-A_i-1}\prod_{i\notin v}\min(u_i,1-u_i)^{-A_i}
	\end{equation}
	holds for some $A_i>0$, $B<\infty$ and all $v\subseteq 1{:}s$.	
\end{definition}

\begin{proposition}\label{prop:owen}
	Let $g$ be a function defined over $(0,1)^s$, and let $\bm u_1,\dots,\bm u_n$ be the first $n$ points of a scrambled $(t,s)$-sequence in base $b\ge 2$.  If $g$ satisfies the boundary growth condition \eqref{eq:grow} with $\max_i A_i<1$, then
	\begin{equation}\label{eq:prop1}
	\E{\abs{\frac{1}{n}\sum_{i=1}^ng(\bm u_i)-\int_{(0,1)^s}g(\bm u)\mrd \bm u}}=O(n^{-1+\max_{i\in 1{:}d}A_i+\epsilon}),
	\end{equation}
	for arbitrarily small $\epsilon>0$.
\end{proposition}
\begin{proof}
	See Theorem 5.7 of \cite{owen:2006}.
\end{proof}

\begin{remark}\label{rem:scrambling}
	As remarked in \cite{owen:2006}, the rate in \eqref{eq:prop1} holds as well for the space efficient alternative scrambling proposed by Matou{\v{s}}ek \cite{mato:1998}. As a result, the mean error rates established in the following also hold for the scrambling method of Matou{\v{s}}ek \cite{mato:1998}, which is implemented in the toolbox of MATLAB. 
\end{remark}
\begin{remark}\label{rem:rate}
	If all the growth rates $A_i$ are arbitrarily small, we arrive at the optimal rate $O(n^{-1+\epsilon})$. The constant $\epsilon$ in \eqref{eq:prop1} is used for  hiding the logarithmic term $(\log n)^s$, which depends on the dimension $s$ of the problem. As a result, the dimension $s$ may has an important impact on the QMC efficiency, even for the optimal case of the growth rates.
\end{remark}

The boundary growth condition is critical in establishing the error rate for smooth integrands with singularities at the boundary of the unit cube. It actually requires the existence of the mixed partial derivatives of the integrands. By the chain rule, we have 
\begin{equation*}
D_{v}q_j(\bm u_{-j})=(D_{v}P_j f)(\bm y)\prod_{i\in v}\frac{\mrd \Phi^{-1}(u_i)}{\mrd u_i},
\end{equation*}
where $\bm y=\bm x_{-j}=\Phi^{-1}(\bm u_{-j})$.
We next give conditions such that the mixed partial derivatives $D_{v}P_j f$ exist for any $v\subseteq 1{:}d\backslash \{j\}$ and then the boundary growth condition for $q_j$ holds.

\begin{assumption}\label{assum:grc}
	Suppose that the integers $j\in 1{:}d$ and $r\ge 1$ are fixed. There exist constants $L>0$ and $B_i\in (0,1)$ such that 
	\begin{align}
	\abs{(D^{\bm \alpha}g)(\bm x)}&\leq L\prod_{i=1}^d \min(\Phi(x_i),1-\Phi(x_i))^{-B_i},\label{eq:dgup}\\
	\abs{(D^{\bm \alpha}\phi)(\bm x)}&\leq L\prod_{i=1}^d \min(\Phi(x_i),1-\Phi(x_i))^{-B_i},\label{eq:dphi}\\	
	\abs{(D_j\phi)(\bm x)}^{-b}&\leq L\prod_{i=1}^d \min(\Phi(x_i),1-\Phi(x_i))^{-B_i}\label{eq:dphib}
	\end{align}
	hold for $1\le b\le 2r-1$ and any multi-index $\bm \alpha$ satisfying $\abs{\bm \alpha}\le r$.	
\end{assumption}

The parameter $r$ can be viewed as a measure of the smoothness. Condition~\eqref{eq:dgup} ensures that $\E{\abs{f(\bm x)}}<\infty$ because all $B_i<1$.
We show in Section~\ref{sec:ex} that Assumption~\ref{assum:grc} is satisfied with arbitrarily small $B_i>0$ for several typical examples from financial engineering.

\begin{assumption}\label{assum:tails}
	Let $\rho\in \mathcal{C}^{r-1}(\R)$ be a strictly positive PDF, and let $\Phi(x)=\int_{-\infty}^x\rho(x)\mrd x$ be the associated CDF. Assume that there exits constants $B,L'>0$ such that \begin{equation}\label{eq:finvbound}
	\frac{\mrd \Phi^{-1}(u)}{\mrd u}\le L'\min(u,1-u)^{-1-B}.
	\end{equation}
	Assume also that for any nonnegative integer $c\le r-1$ and any $a\in(0,1)$,
	\begin{equation}\label{eq:lim1}
	\lim_{x\to \infty}\frac{\rho^{(c)}(x)}{(1-\Phi(x))^a}=0, \text{ and}
	\end{equation}
	\begin{equation}\label{eq:lim2}
	\lim_{x\to -\infty}\frac{\rho^{(c)}(x)}{\Phi(x)^a}=0.
	\end{equation}
\end{assumption}

The next lemma shows that Assumption~\ref{assum:tails} is  satisfied with arbitrarily small $B>0$ for the standard normal distribution. 

\begin{lemma}\label{lem:normal}
	If $\rho$ is the density of the standard normal distribution, i.e., $$\rho(x)=\frac{1}{\sqrt{2\pi}}\exp(-x^2/2),$$ then Assumption~\ref{assum:tails} is satisfied with arbitrarily small $B>0$ and any $r\geq 1$.
\end{lemma}
\begin{proof}
	It is easy to see that $\rho\in\mathcal{C}^\infty(\R)$ and $\rho(x)>0$ for all $x\in\R$. Let $\Phi$ be the CDF of the standard normal distribution. Note that
	\begin{equation}\label{eq:norminv}
	\begin{cases}
	\Phi^{-1}(\epsilon)&=-\sqrt{-2\log(\epsilon)}+o(1)\\
	\Phi^{-1}(1-\epsilon)&=\sqrt{-2\log(\epsilon)}+o(1)
	\end{cases}
	\end{equation}
	as $\epsilon\downarrow 0$ (see Chapter 3.9 of \cite{pete:read:1996}). We fine that
	\begin{align*}
	\frac{\mrd \Phi^{-1}(u)}{\mrd u}=\frac{1}{\rho(\Phi^{-1}(u))}= \sqrt{2\pi}\exp(\Phi^{-1}(u)^2/2)
	\end{align*}
	For any $B>0$, we have
	\begin{align*}
	\lim_{u\to 0+}\frac{\mrd \Phi^{-1}(u)}{\mrd u}u^{1+B}&=\lim_{u\to 0+}\sqrt{2\pi}\exp\{[-\sqrt{-2\log(u)}+o(1)]^2/2+(1+B)\log(u)\}\\&=\lim_{u\to 0+}\sqrt{2\pi}\exp\{B\log(u)-o(\sqrt{-2\log(u)})\}\\
	&=\lim_{u\to 0+}\sqrt{2\pi}\exp\{-(B/2)[\sqrt{-2\log(u)}+o(1)]^2\}\\&=0,
	\end{align*}
	and similarly,
	$$\lim_{u\to 1-}\frac{\mrd \Phi^{-1}(u)}{\mrd u}(1-u)^{1+B}=0.$$
	This gives $$\frac{\mrd \Phi^{-1}(u)}{\mrd u}=O(\min(u,1-u)^{-1-B})$$
	for any $B>0$. Gordon \cite{gord:1941} showed that $1-\Phi(x)>\rho(x)/(x+1/x)$ for $x>0$. For $a\in (0,1)$ and any nonnegative integer $k$, we have
	\begin{equation*}
	\lim_{x\to \infty}\frac{x^k\rho(x)}{(1-\Phi(x))^a}\le \lim_{x\to \infty}\frac{x^k\rho(x)}{(\rho(x)/(x+1/x))^a}=\lim_{x\to \infty}x^k(x+1/x)^a\rho(x)^{1-a}=0.
	\end{equation*}
	Note that $\rho^{(c)}(x)$ is a linear combination of some terms of the form $x^k\rho(x)$ with $k\le c$. We therefore obtain \eqref{eq:lim1}. The equality \eqref{eq:lim2} can be obtained by replacing $x$ with $-x$ in \eqref{eq:lim1}.
\end{proof}

\begin{theorem}\label{eq:mainj}
	Suppose that Assumptions~\ref{assump1}, \ref{assum:grc} and \ref{assum:tails}  are satisfied with fixed constants $B_i\in(0,1)$, $B>0$, $r\ge 1$ and $j\in1{:}d$.  Suppose that $f$ is given by \eqref{eq:integrand} with $g,\phi\in \mathcal{C}^r(\R^d)$.
	\begin{itemize}
		\item If $B_j<1/(2r+1)$, then $P_j f\in\mathcal{C}^r(\R^{d-1})$.
		\item Suppose that $r\ge d-1$ and $B_j<1/(2d-1)$. If $$\gamma_j:=(2d-1)\max_{i\in 1{:}d\backslash\{j\}} B_i+B<1,$$ then 
		\begin{equation}\label{eq:ratej}
		\E{|\hat{I}(P_j f)-{I}(f)|}=O(n^{-1+\gamma_j
			+\epsilon})
		\end{equation}
		for arbitrarily small $\epsilon>0$. 
		\item  Suppose that $r\ge d-1$. If $B_1,\dots,B_d$ and $B$ are arbitrarily small, then 
		\begin{equation*}\label{eq:optratej}
		\E{|\hat{I}(P_j f)-{I}(f)|}=O(n^{-1+\epsilon}).
		\end{equation*}
	\end{itemize}    
\end{theorem}
\begin{proof}
	Let $\bm y=\bm x_{-j}$. We first prove that for any $\bm y^*\in \R^{d-1}$, there exists a set $B(\bm y^*,\delta)$ with $\delta>0$ such that  $\int_{-\infty}^\infty D^{\bm \alpha} g(
	x_j,\bm y)\rho(x_j)\mrd x_j$ converges uniformly on the ball $B(\bm y^*,\delta)$ for any multi-index $\bm{\alpha}$ with $\abs{\bm{\alpha}}\le r$ and $\alpha_j=0$. For any $\delta>0$, from \eqref{eq:dgup}, we have
	\begin{equation*}
	\sup_{\bm y\in B(\bm y^*,\delta)}\abs{D^{\bm \alpha} g(
		x_j,\bm y)\rho(x_j)}\le M \min(\Phi(x_j),1-\Phi(x_j))^{-B_j}\rho(x_j),
	\end{equation*}
	where $M$ is a constant depending on $\delta$. Together with  $$\int_{-\infty}^{\infty}\min(\Phi(x_j),1-\Phi(x_j))^{-B_j}\rho(x_j)\mrd x_j<\infty,$$ Weierstrass test admits that $\int_{-\infty}^\infty D^{\bm \alpha} g(
	x_j,\bm y)\rho(x_j)\mrd x_j$ converges uniformly on $B(\bm y,\delta)$.

	For the function $h$ given by \eqref{eq:hy}, by 　Assumption~\ref{assum:grc}, we find that
	\begin{equation}\label{eq:hbound}
	\abs{h(\bm y)} \le\kappa(\psi(\bm y))\prod_{i\in 1{:}d\backslash \{j\}} \min(\Phi(x_i),1-\Phi(x_i))^{-(2r+1)B_i},
	\end{equation} 
	where 
	\begin{equation*}
	\kappa(x)=\abs{\beta}L^{2r+1}\min(\Phi(x),1-\Phi(x))^{-(2r+1)B_j}\rho^{(c)}(x).
	\end{equation*}
	Again $\psi(\bm y)\to \pm \infty$ as $\bm y$ approaches a boundary point of $U_j$ lying in $U_j^-$ and $U_j^+$, respectively. Since $(2r+1)B_j<1$,  it follows \eqref{eq:lim1} and \eqref{eq:lim2} that $\kappa(x)\to 0$ as $x\to \pm\infty$, leading to \eqref{eq:bounds}.
	Therefore, by Theorem~\ref{thm:grie}, we have $P_j f\in\mathcal{C}^{r}(\R^{d-1})$.

	Using \eqref{eq:dgup} again gives
	\begin{align*}
	\int_{-\infty}^{\infty}\abs{(D^{\bm\alpha} g)(x_j,\bm x_{-j})}\rho(x_j)\mrd x_j\le \frac{2^{B_j}L}{1-B_j}\prod_{i\in 1{:}d\backslash \{j\}} \min(\Phi(x_i),1-\Phi(x_i))^{-B_i},\label{eq:dgbound}
	\end{align*}
	where we used $B_j<1$. 
	Since $\kappa(x)$ is continuous over $\R$, $\kappa(x)$ is bounded. 
	According to \eqref{eq:hbound} (replacing $r$ with $\abs{\bm{\alpha}}$), 
	\begin{equation*}
	\abs{h_{\bm \alpha,i}(\bm y)} \le  M'\prod_{i\in 1{:}d\backslash \{j\}} \min(\Phi(x_i),1-\Phi(x_i))^{-(2\abs{\bm{\alpha}}+1)B_i}
	\end{equation*} 
	for finite $M'>0$.
	By \eqref{eq:bound4dfj}, we have
	\begin{equation}\label{eq:dpfbound}
	\abs{(D^{\bm \alpha} P_j f)(\bm x_{-j})}\le \tilde{L}\prod_{i\in 1{:}d\backslash \{j\}} \min(\Phi(x_i),1-\Phi(x_i))^{-(2\abs{\bm{\alpha}}+1)B_i},
	\end{equation}
	where $\tilde{L}$ is a constant. 
	
	Now assume $r\ge d-1$, and let $v\subset 1{:}d\backslash \{j\}$. Let $\bm \alpha$ be a multi-index with entries $\alpha_i=1$ for $i\in v$ and $\alpha_i=0$ otherwise. Then $\abs{\bm \alpha}=\abs{v}\le d-1$ and $D_v P_j f=D^{\bm\alpha}P_jf$.
	Let $q_j(\bm u_{-j})=P_j f(\Phi^{-1}(\bm u_{-j}))$. Using \eqref{eq:finvbound} and \eqref{eq:dpfbound}, we obtain
	\begin{align*}
	\abs{(D_vq_j)(\bm u_{-j})}&=\abs{(D_v P_j f)(\Phi^{-1}(\bm u_{-j}))}\prod_{i\in v} \frac{\mrd \Phi^{-1}(u_i)}{\mrd u_i}\\
	&\le \tilde{L}L'\prod_{i\in 1{:}d\backslash \{j\}} \min(u_i,1-u_i)^{-(2d-1)B_i}\prod_{i\in v}\min(u_i,1-u_i)^{-1-B}.
	\end{align*}
	As a result, $q_j(\bm u_{-j})$  satisfies the boundary growth condition with rates $$A_i=(2d-1)B_i+B\in (0,1),$$ for all $i\neq j$. Note that $I(P_j f)=I(f)$. Applying Proposition~\ref{prop:owen} then gives \eqref{eq:ratej}. Finally, letting all $B_i$ and $B$ be arbitrarily small, the last part holds immediately.
\end{proof}

The rate in \eqref{eq:ratej} suggests that if there is a list of candidates $x_j$ to be integrated out, we prefer to choose the one with largest $B_j$ because it delivers the best rate.

\section{ANOVA Decomposition}\label{sec:anova}
Griebel et al.~\cite{grie:2013} investigated the smoothness property for the terms of the ANOVA decomposition of functions with kink. In this section, we study the convergence rate for the RQMC integration with the ANOVA terms of discontinuous functions. The ANOVA decomposition of $f$ is given by 
\begin{equation*}
f(\bm x)=\sum_{v\subseteq 1{:}d} f_v(\bm x),
\end{equation*}
where $f_v(\bm x)$ depends only on the variables $x_j$ with indices $j\in v$, and satisfies $P_jf_v\equiv0$ for all $j\in v$.

For $v\subseteq 1{:}d$, $\bm x_v$ denotes the components of $\bm x$ with indices in $v$. Let $-v=1{:}d\backslash v$ and let $\abs{v}$ be the cardinality of the set $v$.
Generally, one can integrate \eqref{eq:integrand} with respect to $x_j$ with $j\in v$, that is,
\begin{equation*}\label{eq:condv}
(P_v f)(\bm x_{-v}):=\E{f(\bm x)|\bm x_{-v}}=\int_{\R^{\abs{v}}} f(\bm x_v,\bm x_{-v})\prod_{i\in v}\rho(x_i)\mrd \bm x_v.
\end{equation*}
We may write that $P_v=\prod_{j\in v}P_j$. Fubini's theorem allows us to take any order within the product. 
The ANOVA terms are defined through the recurrence relation
\begin{equation*}
f_v = P_{-v}f-\sum_{w\subsetneq v}f_w,
\end{equation*}
where $f_{\varnothing}=I(f)$ by convention. Kuo et al. \cite{kuo:2010} showed that the ANOVA terms can expressed explicitly by
\begin{equation}\label{eq:fanova}
f_v = \sum_{w\subseteq v} (-1)^{\abs{v}-\abs{w}}P_{-w}f.
\end{equation} 

We next pay particular attention to RQMC integration with $P_v f$ for general $v\subseteq 1{:}d$. The study of $\hat{I}(P_v f)$ below paves the way to understand the QMC error of ANOVA components of the integrand, although the projection $P_v f$ cannot be calculated analytically in practice. Note that $I(P_v f)=I(f)$.

\begin{theorem}\label{thm:pvf}
	Consider the setup in Theorem~\ref{eq:mainj}. Let $v\subseteq 1{:}d$ satisfying $j\in v$.
	\begin{itemize}
		\item If $\max_{i\in v} B_i<1/(2r+1)$, then $P_v f\in \mathcal{C}^r(\R^{d-\abs{v}})$.
		\item Suppose that $r\ge d-1$ and $\max_{i\in v} B_i<1/(2d-1)$.  If $$\gamma_v:=(2d-1)\max_{i\in 1{:}d\backslash v} B_i+B<1,$$ then 
		\begin{equation}\label{eq:ratev}
		\E{|\hat{I}(P_v f)-{I}(f)|}=O(n^{-1+\gamma_v
			+\epsilon})
		\end{equation}
		for arbitrarily small $\epsilon>0$.
		\item Suppose that $r\ge d-1$. If $B_1,\dots,B_d$ and $B$ are arbitrarily small, then \eqref{eq:ratev} holds with $\gamma_v=0$.
	\end{itemize} 
\end{theorem}
\begin{proof}
	From Theorem~\ref{eq:mainj}, we have $P_j f\in\mathcal{C}^r(\R^{d-1})$. Now suppose that there is $k\in v$ satisfying $k\neq j$. Let $\bm{\alpha}$ be any multi-index  with $\abs{\bm{\alpha}}\le r$ and $\alpha_i=0$ for all $i\in \{j,k\}$. For any $\delta>0$ and any $\bm x_{-\{j,k\}}^*\in \R^{d-2}$, it follows from \eqref{eq:dpfbound} that 
	\begin{equation*}
	\sup_{\bm x_{-\{j,k\}}\in B(\bm x_{-\{j,k\}}^*,\delta)}\abs{(D^{\bm \alpha} P_j f)( x_k,\bm x_{-\{j,k\}})}\le M \min(\Phi(x_k),1-\Phi(x_k))^{-(2\abs{\bm{\alpha}}+1)B_k},
	\end{equation*}
	where $M$ is a constant depending on $\delta$.  Applying  Weierstrass test again gives that 
	$$(P_kD^{\bm \alpha}P_j f)(\bm x_{-\{j,k\}})=\int_{-\infty}^\infty (D^{\bm \alpha}P_j f)(x_k,\bm x_{-\{j,k\}}) \rho(x_k) \mrd  x_k$$ converges uniformly on the ball $B(\bm x_{-\{j,k\}}^*,\delta)$. By Theorem~\ref{thm:changes}, we have $D^{\bm \alpha}P_kP_j f=P_kD^{\bm \alpha}P_j f$ and hence $P_kP_j f\in \mathcal{C}^r(\R^{d-2})$. By \eqref{eq:dpfbound},
	\begin{align*}
|(D^{\bm \alpha}P_kP_j f)&(\bm x_{-\{j,k\}})|=\abs{(P_kD^{\bm \alpha}P_j f)(\bm x_{-\{j,k\}})}\\&\le \tilde{L}P_k \prod_{i\in 1{:}d\backslash \{j\}} \min(\Phi(x_i),1-\Phi(x_i))^{-(2\abs{\bm{\alpha}}+1)B_i}\\
	&=\tilde{L}\frac{2^{(2\abs{\bm{\alpha}}+1)B_k}}{1-(2\abs{\bm{\alpha}}+1)B_k} \prod_{i\in 1{:}d\backslash \{j,k\}} \min(\Phi(x_i),1-\Phi(x_i))^{-(2\abs{\bm{\alpha}}+1)B_i},
	\end{align*}
	where we used $(2\abs{\bm{\alpha}}+1)B_k<1$.
	
	Performing the same procedure recursively on all other  elements in the set $v$, one can easily show that $D^{\bm \alpha}P_v f=P_vD^{\bm \alpha} f$ and
	\begin{equation}\label{eq:dapvf}
	\abs{(D^{\bm \alpha}P_v f)(\bm x_{-v})}\le \tilde{L}\prod_{i\in v\backslash\{j\}}\frac{2^{(2\abs{\bm{\alpha}}+1)B_i}}{1-(2\abs{\bm{\alpha}}+1)B_i}\prod_{i\in 1{:}d\backslash v} \min(\Phi(x_i),1-\Phi(x_i))^{-(2\abs{\bm{\alpha}}+1)B_i},
	\end{equation}
	where  $\bm{\alpha}$ is a multi-index with $\abs{\bm{\alpha}}\le r$ and $\alpha_i=0$ for all $i\in v$. 
	Moreover, $P_vf\in \mathcal{C}^r(\R^{d-\abs{v}})$.

	We now prove the second part.  For any $w\subseteq 1{:}d\backslash v$, by \eqref{eq:dapvf}, we find that 
	\begin{align*}
	\abs{(D_w P_v f)(\bm x_{-v})}\le \tilde{L}\prod_{i\in v\backslash\{j\}}\frac{2^{(2d-1)B_i}}{1-(2d-1)B_i}\prod_{i\in 1{:}d\backslash v} \min(\Phi(x_i),1-\Phi(x_i))^{-(2d-1)B_i}.
	\end{align*}
	Let $q_v(\bm u_{-v})=P_vf(\Phi^{-1}(\bm u_{-v}))$. By \eqref{eq:finvbound}, we obtain
	\begin{align*}
	\abs{(D_wq_v)(\bm u_{-v})}&=\abs{(D_w P_v f)(\Phi^{-1}(\bm u_{-v}))}\prod_{i\in w} \frac{\mrd \Phi^{-1}(u_i)}{\mrd u_i}\\
	&\le L_0\prod_{i\in 1{:}d\backslash v} \min(u_i,1-u_i)^{-(2d-1)B_i}\prod_{i\in w}\min(u_i,1-u_i)^{-1-B}
	\end{align*}
	for some constant $L_0>0$. As a result, $q_v$  satisfies the boundary growth condition with rates $A_i=(2d-1)B_i+B\in (0,1)$ for $i\in 1{:}d\backslash v$. Note that $I(P_v f)=I(f)$. Applying Proposition~\ref{prop:owen} then gives \eqref{eq:ratev}.  Letting all $B_i$ and $B$ be arbitrarily small, the last part holds immediately.
\end{proof}

Theorem~\ref{thm:pvf} shows that $P_vf$ inherits the full smoothness of $g$ and $\phi$. An interesting point behind \eqref{eq:ratev} is that integrating more variables out does not decrease the error rate of RQMC.

\begin{theorem}\label{thm:anovarate}
	Consider the setup in Theorem~\ref{eq:mainj}. Let $f_v$ be the ANOVA term \eqref{eq:fanova} for the function $f$ given by \eqref{eq:integrand}. 
	\begin{itemize}
		\item Suppose that $j\notin v$. If $\max_{i\in 1{:}d} B_i<1/(2r+1)$, then $f_v \in\mathcal{C}^r(\R^{\abs{v}})$. If $r\ge d-1$, $\max_{i\in 1{:}d} B_i<1/(2d-1)$ and $$\gamma_{-v}=(2d-1)\max_{i\in v} B_i+B<1,$$ then
		\begin{equation*}\label{eq:ratetv}
		\E{|\hat{I}(f_v)-{I}(f_v)|}=O(n^{-1+\gamma_{-v}
			+\epsilon})
		\end{equation*} 
		for arbitrarily small $\epsilon>0$.
		\item Suppose that Assumptions~\ref{assump1} and \ref{assum:grc} are satisfied with all $j\in 1{:}d$. Then the results above hold for any $v\subsetneq 1{:}d$. Suppose that $\Omega$ is given by \eqref{eq:omega}, whose boundary $\partial \Omega$ admits a $(d-1)$-dimensional Minkowski content\footnote{See \cite{he:2017} for the formal definition. In the terminology of geometry, the Minkowski content is known as the surface area of the set $\Omega$. Clearly, the convex sets in $[0,1]^d$ satisfy this condition as their surface areas are bounded by the surface area of the unit cube, which is $2d$.}. If $r\ge d-1$, $B_1,\dots,B_d$ and $B$ are arbitrarily small, then 
		\begin{equation}\label{eq:rateanova}
		\E{|\hat{I}(f_v)-{I}(f_v)|}=
		\begin{cases}
		0, &v=\varnothing\\
		O(n^{-1+\epsilon}),&\varnothing\neq v\subsetneq 1{:}d\\
		O(n^{-1/2-1/(4d-2)+\epsilon}), &v=1{:}d.
		\end{cases}
		\end{equation}
	\end{itemize}
\end{theorem}

\begin{proof}
	For any $w\subseteq v$, by Theorem~\ref{thm:pvf}, we have $P_{-w}f\in\mathcal{C}(\R^{\abs{w}})$ and $$\E{|\hat{I}(P_{-w} f)-{I}(P_{-w} f)|}=O(n^{-1+\gamma_{-w}
		+\epsilon}),$$ since $j\in1{:}d\backslash w$. The first part immediately follows from \eqref{eq:fanova} and $\gamma_{-w}\le \gamma_{-v}$. If Assumptions~\ref{assump1} and \ref{assum:grc} are satisfied with all $j\in 1{:}d$, then the results in the first part hold for any $v\subsetneq 1{:}d$. 	
	We now assume that all $B_i$ and $B$ are arbitrarily small. The first two cases in \eqref{eq:rateanova} are straightforward. For the case $v=1{:}d$, $f_{1{:}d}=f-\sum_{v\subsetneq 1{:}d} f_v$. Note that $f(\Phi^{-1}(\bm u))=g(\Phi^{-1}(\bm u))\I{\bm u\in\Omega}$, where $g(\Phi^{-1}(\bm u))$ viewed as a function over $(0,1)^d$ satisfies the boundary growth condition with rates $A_i=B_i+B$. By Corollary 3.5 of \cite{he:2017}, we obtain that 
	\begin{equation*}
	\E{|\hat{I}(f)-{I}(f)|}=O(n^{-1/2-1/(4d-2)+\epsilon}).
	\end{equation*}
	Using the triangle inequality, we finally have
	\begin{align*}
	\E{|\hat{I}(f_{1{:}d})-{I}(f_{1{:}d})|}&\leq \E{|\hat{I}(f)-{I}(f)|}+ \sum_{\varnothing \neq v\subsetneq 1{:}d}\E{|\hat{I}(f_v)-{I}(f_v)|}\\&=O(n^{-1/2-1/(4d-2)+\epsilon}),
	\end{align*}
	which completes the proof.
\end{proof}

Theorem~\ref{thm:anovarate} suggests that QMC can still be very effective for non-smooth integrands if they  have low effective dimension. For these cases, QMC integration with all the ANOVA terms of the non-smooth integrand, expect the one of highest order, can enjoy the best possible rate $O(n^{-1+\epsilon})$. The highest non-smooth term contributes little to the integration error as its variance is negligible compared to the total variance. This finding can explain the success of  dimension reduction techniques used in improving the accuracy of QMC, such as the linear transform (LT) method proposed by Imai and Tan \cite{imai:tan:2006}, and the gradient principal component analysis (GPCA) method proposed by Xiao and Wang \cite{xiao:wang:2017}. 

\section{Examples from Option Pricing and Greeks Estimation}\label{sec:ex}
\subsection{Model Setting}
Let $S(t)$ denote the underlying price dynamics at time $t$ under the risk-neutral measure. In a simulation framework, it is common that the prices are simulated at discrete times $t_1,\dots,t_d$ satisfying $0=t_0<t_1<\dots<t_d=T$, where $T$ is the maturity of the financial derivative of interest. Without loss of generality,  we assume that the discrete times are evenly spaced, i.e., $t_i=i\Delta t$, where $\Delta t=T/d$. Denote $S_i=S(t_i)$. We assume that under the risk-neutral measure the asset follows the geometric Brownian motion
\begin{equation}
\frac{\mrd S(t)}{S(t)}=\mu\mrd t+\sigma \mrd B(t),\label{GBM}
\end{equation}
where $\mu$ is
the riskfree interest rate, $\sigma$ is the volatility and $B(t)$
is the standard Brownian motion. Under this framework, the solution of \eqref{GBM} is analytically available
\begin{equation}
S(t)=S_0\exp[(\mu-\sigma^2/2)t+\sigma B(t)],\label{Solution}
\end{equation}
where $S_0$ is the initial price of the asset. Let $\bm{B}:=(B(t_1),\dots,B(t_d))^{\top}$. We have $\bm{B}\sim N(\bm{0},\bm{\Sigma})$, where  $\bm{\Sigma}$ is a positive definite matrix with entries $\Sigma_{ij}=\Delta t\min(i,j)$.

Let $\bm{A}$ be a matrix satisfying $\bm{A}\bm{A}^\top=\bm{\Sigma}$. Using the transformation $\bm{B}=\bm{A}\bm x$, where $\bm x\sim N(\bm 0,\bm I_d)$, it follows from
\eqref{Solution} that
\begin{align}\label{eq:si}
S_i=S_i(\bm{x})= S_0\exp\left[(\mu-\sigma^2/2)i\Delta t+\sigma \sum_{j=1}^da_{ij}x_j\right].
\end{align}
The matrix $\bm A$ is  called the \emph{generating matrix} as it determines the way of simulation. Under the risk-neutral measure, the price and the sensitivities of the financial derivative can 
be expressed as an expectation  $I(f)=\E{f(\bm{x})}$ for a real-valued function $f$ over $\R^d$. It is known that the choice of the matrix $\bm A$ may have an impact on the efficiency of QMC (see, e.g., \cite{he:wang:2014,imai:tan:2006}), but it does not affect the variance of plain MC estimate. 

Many functions in the pricing and hedging of financial derivatives are discontinuous or unbounded, which can be expressed in the form \eqref{eq:integrand}. We next consider some representative examples
of this form. Examples~\ref{example1}--\ref{example6} below are the arithmetic Asian option and its Greeks, which were also studied in \cite{he:2017,weng:2017}. Example~\ref{example7} is the binary option, which was considered in \cite{grie:2017}. The Greeks of the binary option can be treated as those of  the arithmetic Asian option in a similar way, so we omit these cases. In this section, $\rho$ and $\Phi$ denote the PDF and the CDF of the standard normal distribution, respectively.

\begin{example}\label{example1}  
	The discounted payoff of an arithmetic average Asian option is
	\begin{equation}\label{eq:asianpayoff}
	f(\bm{x})=e^{-\mu T}\max(S_A-K,0)= e^{-\mu T}\left(S_A-K\right)\I{S_A\ge K},  
	\end{equation}
	where $S_A=(1/d)\sum_{i=1}^{d} S_i(\bm x)$ and $K$ is the strike price.
\end{example}
\begin{example}\label{example2} 	
	The pathwise estimate of the \textit{delta} of the  Asian option   is
	\begin{equation}\label{asian delta}
	f(\bm{x}) =e^{-\mu T} \frac{S_A}{S_0} \I{S_A\ge K}. 
	\end{equation}
	The  \textit{delta} of an  option  is the sensitivity with respect   to the initial price $S_0$ of the underlying asset.
\end{example}
\begin{example}\label{example3} 
	An estimate of the \textit{gamma} of  the  Asian option is
	\begin{equation}\label{eq:asiangamma}
	f(\bm{x}) =e^{-\mu T}  \frac{S_A \left( \log (S(t_1)/S_0)-(\mu+\sigma^2/2 )\Delta t\right) }{S_0^2\sigma^2 \Delta t  }  \I{S_A\ge K},
	\end{equation}
	which results from applying the pathwise method first and then the likelihood ration method (see \cite{glas:2004}). 	
	The \textit{gamma} is the second derivative with respect to the initial price $S_0$ of the underlying asset. 
\end{example}
\begin{example}\label{example4} 
	The pathwise estimate of the \textit{rho} of the  Asian option is
	\begin{equation}\label{eq:asianrho}
	f(\bm{x}) = e^{-\mu T}\left[\frac{\mrd S_A}{\mrd r} -T(S_A-K)\right]
	\I{S_A\ge K},
	\end{equation}
	where
	\begin{equation*}
	\frac{\mrd S_A}{\mrd r} =\frac{T}{d^2} \left( \sum_{j=1}^{d} j S(t_j)\right).
	\end{equation*}
	The  \textit{rho} of an   option  is the sensitivity with respect   to the risk-free interest rate  $r$.
	
\end{example}
\begin{example}\label{example5} 
	The pathwise estimate of the \textit{theta} of the  Asian option   is
	\begin{equation}\label{eq:asiantheta}
	f(\bm{x}) = e^{-\mu T} \left[\frac{d S_A}{d T}-\mu(S_A-K) \right]
	\I{S_A\ge K},
	\end{equation}
	where
	\begin{equation*}
	\frac{\mrd S_A}{\mrd T} = \frac{1}{d} \sum_{i=1}^{d}  S(t_i)\left[  \frac{\omega i }{2d}  + \frac{\log (S(t_i)/S_0)}{2T}  \right].
	\end{equation*}
	The \textit{theta} of an  option  is the sensitivity with respect   to the maturity of the option  $T$.	
\end{example}
\begin{example}\label{example6} 
	The pathwise estimate of the \textit{vega} of the  Asian option is
	\begin{equation}\label{eq:asianvega}
	f(\bm{x}) =e^{-\mu T} \frac{1}{d} \sum_{i=1}^{d} \frac{\mrd S(t_i)}{\mrd\sigma} \I{S_A\ge K}, 
	\end{equation}
	in which
	\begin{equation*}
	\frac{\mrd S(t_i)}{\mrd\sigma}=S(t_i)\frac{1}{\sigma}\left[ \log\left( \frac{S(t_i)}{S_0}\right) -\left( \mu+ \frac{1}{2}\sigma^2 \right) t_i   \right]. 
	\end{equation*}
	The  \textit{vega} of an  option  is the sensitivity with respect   to the volatility $\sigma$.
\end{example}

\begin{example}\label{example7}
	The discounted payoff of a binary Asian option is
	\begin{equation}\label{eq:binarypayoff}
	f(\bm{x})= e^{-\mu T}\I{S_A\ge K}.
	\end{equation}
\end{example}	

\subsection{CQMC Error Rates}
All the examples above fit into the form \eqref{eq:integrand} with $\phi(\bm x)=S_A-K$ and $g(\bm x)$ depending on the examples. It is easy to see that $g,\phi\in \mathcal{C}^{\infty}(\R^d)$. Lemma~\ref{lem:normal} guarantees the validation of Assumption~\ref{assum:tails}. From \eqref{eq:si}, we find that
\begin{equation}\label{eq:dj}
(D_j \phi)(\bm x)=\frac{\sigma}{d}\sum_{i=1}^d(a_{ij}S_i).
\end{equation}
Thanks to all $S_i>0$, Assumption~\ref{assump1} is satisfied if there exists an index $j\in 1{:}d$ such that  
\begin{equation}\label{eq:condj}
a_{ij}\geq 0 \text{ (or } a_{ij}\leq 0) \text{ for } i=1,\dots,d.
\end{equation}
This condition is not void for commonly used constructions of the Brownian motion. For the standard construction and the Brownian bridge construction, all elements $a_{ij}\geq 0$ so that Assumption~\ref{assump1} is satisfied with all $j\in 1{:}d$. For the principal component analysis (PCA) construction, the elements $a_{i1}$ have the same sign so that Assumption~\ref{assump1} is satisfied with $j=1$, but for $j>1$, the elements $a_{ij}$ can have both signs. See \cite{glas:2004,grie:2010} for details on these commonly used constructions.

It remains to verify Assumption~\ref{assum:grc}. For simplicity, we assume that $a_{ij}\geq 0$ for all $i\in 1{:}d$. For any multi-index $\bm \alpha$ (including $\bm \alpha = \bm 0$), it follows from \eqref{eq:si} that
\begin{equation}\label{eq:daphi}
(D^{\bm \alpha}\phi)(\bm{x}) = \frac 1 d\sum_{i=1}^d(D^{\bm \alpha}S_{i})(\bm{x})-K\I{\bm \alpha = \bm 0},
\end{equation}
where
\begin{equation}\label{eq:dsi}
(D^{\bm \alpha}S_{i})(\bm{x}) = \sigma^{\abs{\bm \alpha}}S_i\prod_{j\in v_+}a_{ij}^{\alpha_j},
\end{equation}
and $v_+=\{j\in 1{:}d|\alpha_j>0\}$. For any $a\in\R$ and any $B>0$, it follows from \eqref{eq:norminv} that 
\begin{align*}
\lim_{u\to 0+}\exp[a\Phi^{-1}(u)]u^{B}
&=\lim_{u\to 0+}\exp[-a\sqrt{-2\log(u)}+B\log (u)]\\
&=\lim_{u\to 0+}\exp[-(B/2)(\sqrt{-2\log(u)}+a/B)^2+a^2/(2B)]\\
&=0.
\end{align*}
Similarly,
$$\lim_{u\to 1-}\exp[a\Phi^{-1}(u)](1-u)^{B}=0.$$
As a result, $\exp(a\Phi^{-1}(u))=O(\min(u,1-u)^{-B}).$ This implies that
\begin{equation}\label{eq:axi}
\exp(a x_i) = O(\min(\Phi(x_i),1-\Phi(x_i))^{-B_i})
\end{equation}
for arbitrarily small $B_i>0$. It then follows from \eqref{eq:si} and \eqref{eq:daphi}--\eqref{eq:axi} that the condition \eqref{eq:dphi} holds for arbitrarily small $B_i>0$ and arbitrary multi-index $\bm \alpha$. 
Since $\bm\Sigma$ is nonsingular, there exists an index $i^*\in 1{:}d$ such that $a_{i^*j}>0$.
For any $b>0$, by \eqref{eq:dj} and \eqref{eq:axi}, we find that
\begin{align*}
|(D_j \phi)(\bm x)|^{-b}&=(\sigma/d)^{-b}\left[\sum_{i=1}^d(a_{ij}S_i)\right]^{-b}\\
&\le (\sigma a_{i^*j})^{-b}d^{b}S_{i^*}^{-b}\\
&=(\sigma a_{i^*j}S_0)^{-b}d^{b}\exp\left[-b(\mu-\sigma^2/2)i^*\Delta t-b\sigma \sum_{j=1}^da_{i^*j}x_j\right]\\
&=O\left(\prod_{i=1}^{d}\min(\Phi(x_i),1-\Phi(x_i))^{-B_i}\right).
\end{align*}
So the condition \eqref{eq:dphib} holds for arbitrarily small $B_i$ and all $b> 0$.  He \cite{he:2017} showed that the boundary growth condition for $g(\Phi^{-1}(\bm{u}))$ is satisfied with arbitrarily small $B_i>0$ for Examples~\ref{example1}--\ref{example6}.
By the same way, it is easy to check that the condition \eqref{eq:dgup} for the function $g(\bm x)$ is satisfied with arbitrarily small $B_i>0$ and arbitrary $\bm \alpha$. For Example~\ref{example7}, the condition \eqref{eq:dgup} is straightforward since $g(\bm x)$ is a constant. As a result, Assumption~\ref{assum:grc} is satisfied with arbitrarily small $B_i>0$ and all $r>0$. Based on the analysis above, we conclude the following theorem for our examples as consequences of Theorems~\ref{thm:pvf} and \ref{thm:anovarate}.

\begin{theorem}\label{thm:examples}
	Let $f(\bm{x})=g(\bm x)\I{\phi(\bm x)\ge 0}$ be one of the functions \eqref{eq:asianpayoff}--\eqref{eq:binarypayoff}, and let $j\in 1{:}d$ be a fixed index satisfying the condition  \eqref{eq:condj}. If $v,w$ are  subsets of $1{:}d$ satisfying $j\in v$ and $j\notin w$, then we have
	\begin{itemize}
		\item
		$P_v f \in  \mathcal{C}^{\infty}(\R^{-v})$, $f_w \in  \mathcal{C}^{\infty}(\R^{w})$, and
		\item 	for arbitrarily small $\epsilon>0$,
		\begin{align*}
		\E{|\hat{I}(P_vf)-{I}(f)|}&=O(n^{-1+\epsilon}),\\
		\E{|\hat{I}(f_w)-{I}(f_w)|}&=O(n^{-1+\epsilon}).
		\end{align*}	
	\end{itemize}
	If the condition \eqref{eq:condj} is satisfied with all $j\in 1{:}d$, the results above hold for all $v\neq \varnothing$ and  $w\neq 1{:}d$. 
\end{theorem}

Note that for the standard construction and the Brownian bridge construction, the condition \eqref{eq:condj} is satisfied with all $j\in 1{:}d$. For these cases, it is not surprising that the ANOVA terms can have unlimited smoothness, except the one of highest order, because Griebel et al. \cite{grie:2013,grie:2016} have shown such a smoothness property for the arithmetic Asian option (Example~\ref{example1}). We extend their results to  discontinuous functions so that the smoothing effect of conditioning can be examined for the Greeks of the arithmetic Asian option (Examples \ref{example2}--\ref{example6}) and the binary option (Example~\ref{example7}). More importantly, we show additionally that QMC can achieve the best possible error rate of $O(n^{-1+\epsilon})$ for these smooth terms with singularities. For the highest order term which is non-smooth, the rate may be just $O(n^{-1/2-1/(4d-2)+\epsilon})$, as claimed in Theorem~\ref{thm:anovarate}.

We now consider the case in which the condition \eqref{eq:condj} (or equivalently Assumption~\ref{assump1}) does not hold. In other words, there exist $i_1,i_2\in 1{:}d$ such that $a_{i_1j}<0<a_{i_2j}$.
For our examples, we find that
\begin{equation*}\label{eq:djj}
(D_jD_j \phi)(\bm x)=\frac{\sigma^2}{d}\sum_{i=1}^d(a_{ij}^2S_i)>0,
\end{equation*}
implying that $\phi(\bm x)$ is strictly convex over $\R^d$. Since $a_{i_1j}<0<a_{i_2j}$, by the implicit function theorem again, there exists a unique function $\Psi\in\mathcal{C}^{\infty}(\R^{d-1})$ such that $(D_j \phi)(\Psi(\bm y),\bm y)=0$ for all $\bm y\in \R^{d-1}$. Therefore, for a given $\bm y$, $\phi(x_j,\bm y)$ is decreasing with respect to $x_j$ for $x_j<\Psi(\bm y)$, while it is increasing for $x_j>\Psi(\bm y)$. This gives
\begin{equation*}
\min_{x_j\in \R}\phi(x_j,\bm y)=\phi(\Psi(\bm y),\bm y).
\end{equation*}
Denote
\begin{align*}
A:&=\{\bm y\in \R^{d-1}|\phi(\Psi(\bm y),\bm y)< 0\},\\
A^c:&=\{\bm y\in \R^{d-1}|\phi(\Psi(\bm y),\bm y)\ge 0\}=\R^{d-1}\backslash A,\text{ and}\\
B:&=\{\bm y\in \R^{d-1}|\phi(\Psi(\bm y),\bm y)= 0\}.
\end{align*}
It is easy to see that $A^c\neq \varnothing$ and $\phi(x_j,\bm y)\ge 0$ for any $\bm y\in A^c$ and any $x_j\in \R$. By the implicit function theorem, there exist two unique functions $\psi_j^L,\psi_j^R\in\mathcal{C}^{\infty}(A)$ such that $\psi_j^L(\bm y)<\Psi(\bm y)<\psi_j^R(\bm y)$ and $\phi(\psi_j^L(\bm y),\bm y)=\phi(\psi_j^R(\bm y),\bm y)=0$ for all $\bm y\in A$. The function $P_j f$ can then be rewritten as
\begin{equation*}
(P_j f)(\bm y)=\displaystyle\begin{cases}
\displaystyle\int_{-\infty}^{\infty}g(x_j,\bm y)\rho(x_j)\mrd x_j,&\ \bm y\in A^c\\
\displaystyle\int_{-\infty}^{\psi_j^L(\bm y)}g(x_j,\bm y)\rho(x_j)\mrd x_j+\int_{\psi_j^R(\bm y)}^{\infty}g(x_j,\bm y)\rho(x_j)\mrd x_j,&\ \bm y\in A.
\end{cases}
\end{equation*}
Note that $\psi_j^L(\bm y)\to \Psi(\bm y)$ and $\psi_j^R(\bm y)\to \Psi(\bm y)$ as $\bm y$ approaches a boundary point of $A$ lying in $A^c$  (i.e., the set $B$). This implies that $P_j f$ is continuous over $\R^{d-1}$. However, for $k\neq j$, $D_kP_j f$ may be no longer continuous over $\R^{d-1}$. Notice that for $\bm y\in \text{interior}(A)$, 
\begin{align*}
(D_kP_j f)(\bm y)&=\int_{-\infty}^{\psi_j^L(\bm y)}(D_k  g)(x_j,\bm y)\rho(x_j)\mrd x_j+\int_{\psi_j^R(\bm y)}^{\infty}(D_k  g)(x_j,\bm y)\rho(x_j)\mrd x_j\notag\\
&-g(\psi_j^L(\bm y),\bm y)\rho(\psi_j^L(\bm y))\frac{(D_k   \phi)(\psi_j^L(\bm y),\bm y)}{(D_j\phi)(\psi_j^L(\bm y),\bm y)}\\&+g(\psi_j^R(\bm y),\bm y)\rho(\psi_j^R(\bm y))\frac{(D_k  \phi)(\psi_j^R(\bm y),\bm y)}{(D_j\phi)(\psi_j^R(\bm y),\bm y)}.
\end{align*}
Let  $\bm y^*$ be a boundary point of of $A$ lying in $A^c$. Then $\bm y^*\in B$. If 
\begin{equation}\label{eq:nonzero}
g(\Psi(\bm y^*),\bm y^*)(D_k \phi)(\Psi(\bm y^*),\bm y^*)\neq 0,
\end{equation}
then
\begin{equation*}\label{eq:dkpjf}
\lim_{\bm y\to \bm y^*}(D_kP_j f)(\bm y)=\infty,
\end{equation*}
because 
\begin{align*}
\lim_{\bm y\to \bm y^*}\frac{1}{(D_j\phi)(\psi_j^L(\bm y),\bm y)}=\lim_{x \to \Psi(\bm y^*)-}\frac{1}{(D_j\phi)(x,\bm y^*)}=-\infty,
\end{align*}
and 
\begin{align*}
\lim_{\bm y\to \bm y^*}\frac{1}{(D_j\phi)(\psi_j^R(\bm y),\bm y)}=\lim_{x\to \Psi(\bm y^*)+}\frac{1}{(D_j\phi)(x,\bm y^*)}=\infty.
\end{align*}
We claim that there exists an index $k\neq j$ such that $(D_k \phi)(\Psi(\bm y^*),\bm y^*)\ne 0$. Otherwise, if $(D_k \phi)(\Psi(\bm y^*),\bm y^*)= 0$ for all $k=1,\dots,d$, we then have all $S_i=0$ by \eqref{eq:dj}. That leads to a contradiction. So the condition \eqref{eq:nonzero} reduces to 
$g(\Psi(\bm y^*),\bm y^*)\neq 0$. For Examples \ref{example2}--\ref{example6} and \ref{example7}, it is easy to see that there exists (at least) a point $\bm y^*\in B$ such that $g(\Psi(\bm y^*),\bm y^*)\neq 0$. This implies that $P_j f\notin  \mathcal{C}^1(\R^{d-1})$ for Examples \ref{example2}--\ref{example6}  and \ref{example7} because the derivative $(D_kP_j f)(\bm y^*)$ does not exist. 

Now let us consider Example~\ref{example1}. For any $\bm y\in A$, since  $$g(\psi_j^L(\bm y),\bm y)=g(\psi_j^R(\bm y),\bm y)=0,$$ $(D_kP_j f)(\bm y)$ reduces to
\begin{equation*}
(D_kP_j f)(\bm y)=\int_{-\infty}^{\psi_j^L(\bm y)}(D_k  g)(x_j,\bm y)\rho(x_j)\mrd x_j+\int_{\psi_j^R(\bm y)}^{\infty}(D_k  g)(x_j,\bm y)\rho(x_j)\mrd x_j,
\end{equation*}
which converges to $\int_{-\infty}^{\infty}(D_k  g)(x_j,\bm y^*)\rho(x_j)\mrd x_j$ as $\bm y\to \bm y^*\in B$. So we have  $P_j f\in  \mathcal{C}^1(\R^{d-1})$. However, the higher order mixed partial derivative $D^{\bm \alpha}P_j f$ with $\abs{\bm \alpha}>1$ does not exist for some point in $B$, because for any $\bm y\in \text{ interior}(A)$ the derivative $D^{\bm \alpha}P_j f$ includes some terms like \eqref{eq:hy} (replacing $\psi$ with  $\psi_j^L$ or $\psi_j^R$), which converges to infinity as $\bm y$ approaches a boundary point of $A$ lying in $A^c$.
This implies that $P_j f\notin  \mathcal{C}^2(\R^{d-1})$ for Example~\ref{example1}. 

We conclude that if the condition \eqref{eq:condj} is violated, $P_j f$ cannot have unlimited smoothness for our examples. It is only guaranteed that $P_j f$ is continuous. The mixed partial derivatives of $(P_j f)\circ \Phi^{-1}$ may have additional
singularities beyond those at the boundary of the unit cube $[0,1]^{d-1}$.
So, in this case, we may not obtain the mean error rate $O(n^{-1+\epsilon})$ for the RQMC estimate $\hat{I}(P_j f)$ as in Theorem~\ref{thm:examples}, except the case of Example~\ref{example1} with $d=2$. This suggests that under the PCA construction, integrating  $x_j$ out for $j>1$ may not gain as much as integrating out $x_1$ in improving the efficiency of QMC. 

Although conditioning can result in smooth integrands, the resulting integrands may have large effective dimension. To circumvent this, a good strategy in practice is to employ some dimension reduction techniques after the conditioning process, such as the LT method proposed by Imai and Tan \cite{imai:tan:2006} and the GPCA method proposed by Xiao and Wang \cite{xiao:wang:2017}. Weng et al. \cite{weng:2017} called this strategy the two-step procedure. In the first step (called the smoothing step), they used the variables push-out smoothing method to remove the discontinuities in the target functions. In the second step (called
the dimension reduction step), they used a so-call CQR method to reduce the effective dimension of the smoothed function. In the the smoothing step, one can use the conditioning technique instead. In principle, the dimension reduction techniques transform  the smooth function $(P_v f)(\bm x_{-v})$
to the form $(P_v f)(\bm U\bm x_{-v})$, where $\bm U$ is a well-chosen orthogonal matrix. The transformation does not affect the unbiasedness of the estimate, since
$$\E{(P_v f)(\bm U\bm x_{-v})}=\E{(P_v f)(\bm x_{-v})}=I(f)$$
holds for arbitrary orthogonal matrix $\bm U$. The next theorem shows that using  dimension reduction techniques after conditioning does not change the smoothness property and the mean error rate.

\begin{theorem}\label{thm:vrt}
	Let $f(\bm{x})=g(\bm x)\I{\phi(\bm x)\ge 0}$ be one of the functions \eqref{eq:asianpayoff}--\eqref{eq:binarypayoff}, and let $j\in 1{:}d$ be a fixed index satisfying the condition  \eqref{eq:condj}. Define $\tilde{f}(\bm x_{-v})=(P_v f)(\bm U \bm x_{-v})$, where $v\subseteq 1{:}d$, $\bm U\in \R^{(d-\abs{v})\times(d-\abs{v})}$ is an arbitrary orthogonal matrix. If $j\in v$, then $\tilde{f}\in \mathcal{C}^{\infty}(\R^{-v})$, and $\mathbb{E}[|\hat{I}(\tilde{f})-{I}(f)|]=O(n^{-1+\epsilon})$	for arbitrarily small $\epsilon>0$. 
\end{theorem}
\begin{proof}
	Note that $(P_v f)(\bm x_{-v})$ depends on the generating matrix $\bm A$ satisfying $\bm A\bm A^\top=\bm \Sigma$. For an arbitrary orthogonal matrix $\bm U$, $(P_v f)(\bm U \bm x_{-v})$ can be expressed as $(P_v f)(\bm x_{-v})$ by replacing the matrix $\bm A$ with another generating matrix $\tilde{\bm A}$ satisfying $\tilde{\bm A}\tilde{\bm A}^\top=\bm \Sigma$. Since the $j$th columns of $\bm A$ and $\tilde{\bm A}$ are the same, the condition  \eqref{eq:condj} still holds for $\tilde{\bm A}$. Applying Theorem~\ref{thm:examples}, we obtain the desired results.
\end{proof}

\subsection{Using CQMC in Practice}
A practical aspect of using CQMC is to calculate analytically $P_j f$. Indeed, if we choose the standard construction of the Brownian motion, it is easy to obtain the closed form of  $P_j f$ with $j=1$ for the examples above. Under the standard construction, we have
\begin{equation*}
\bm{A} = \sqrt{\Delta t}\left[
\begin{array}{cccc}
1 & 0 & \cdots & 0\\
1 & 1 &  \cdots & 0\\
\vdots & \vdots &\ddots  & \vdots\\
1 & 1 & \cdots & 1 
\end{array}
\right],
\end{equation*}
which is a lower triangular matrix. This gives
$$S_i = \exp(\sigma \sqrt{\Delta t} x_1)g_i(\bm x_{2{:}d}),$$
where $g_i(\bm x_{2{:}d})=S_0\exp[(\mu-\sigma^2/2)i\Delta t+\sigma \sqrt{\Delta t} \sum_{j=2}^i x_j]$. It is easy to see that for all $\bm x_{2:d}\in \mathbb{R}^{d-1}$, 
$$\psi_1(\bm x_{2{:}d})=\frac{1}{\sigma \sqrt{\Delta t}}\left[\log(dK)-\log\left(\sum_{i=1}^d g_i(\bm x_{2{:}d})\right)\right].$$
This implies that $U_1 =  \mathbb{R}^{d-1}$, $U_1^+=U_1^-=\varnothing$. 
Therefore, 
\begin{align*}
(P_1 f)(\bm x_{2{:}d}) &= \frac{1}{\sqrt{2\pi}}\int g(\bm x)\I{\phi(\bm x)\ge 0}\exp(-x_1^2/2)\mrd x_1\\
&=\frac{1}{\sqrt{2\pi}}\int_{\psi_1(\bm x_{2{:}d})}^{\infty} g(x_1,\bm x_{2{:}d})\exp(-x_1^2/2)\mrd x_1.
\end{align*}
Since $g(\bm x)$ in the examples is a linear combination of components $S_i$ and $S_i\log (S_k)$, $i,k=1,\dots,d$, it reduces to calculate integrals of the form
$$\mu(a,b,c,\ell) = \frac{1}{\sqrt{2\pi}}\int_{a}^\infty (b+cx_1)\exp(-x_1^2/2+\ell x_1)\mrd x_1,$$
where $a,b,c,\ell$ do not depend on $x_1$. By the change of variables, we have
\begin{align*}
\mu(a,b,c,\ell) &= \frac{\exp(\ell^2/2)}{\sqrt{2\pi}}\int_{a-\ell}^\infty (b+\ell+cx_1)\exp(-x_1^2/2)\mrd x_1\\
&=\exp(\ell^2/2)(b+\ell)[1-\Phi(a-\ell)]+\frac{\exp(\ell^2/2)c}{\sqrt{2\pi}}\exp(-(a-\ell)^2/2).
\end{align*}
Therefore, $P_1 f$ is available for the examples considered in this section. 

Similarly, one can also obtain the closed form of $P_j f$ with $j=d$, for which $S_1,\dots,S_{d-1}$ do not depend on $\bm x_{1:{d-1}}$. For this case, 
$$\psi_d(\bm x_{1{:}d-1})=\frac{1}{\sigma \sqrt{\Delta t}}\left[\log\left(dK-\sum_{i=1}^{d-1} S_i\right)-\log (S_0)-\left(\mu-\frac{\sigma^2}{2}\right)T\right]-\sum_{i=1}^{d-1}x_i,$$
where $\bm x_{1{:}d-1}\in U_d=\{\bm x_{1{:}d-1}|\sum_{i=1}^{d-1} S_i<dK\}$. Also, $U_d^+=\{\bm x_{1{:}d-1}|\sum_{i=1}^{d-1} S_i\ge dK\}=\mathbb{R}^d\backslash U_d$ and $U_d^-=\varnothing$. 

However, for $1<j<d$, $\psi_j(\bm x_{-j})$ cannot be computed analytically under the standard construction. This is the case for the Brownian bridge construction (except the case $j=d$) and the PCA construction with $j=1$. One thus may resort to some root-finding algorithms (such as Newtons' method) to solve the equation $\phi(x_j,\bm{x}_{-j})=0$ for given $\bm x_{-j}\in U_j$. 

In our numerical experiments, we examine the mean error rate of  CQMC for Example~\ref{example2} with the standard construction and $j=1$. From the analysis above, $P_1 f$ is available. We also investigate the combination of the CQMC method with the GPCA method proposed by \cite{xiao:wang:2017}. The combined method is called CQMC+GPCA. Theorem~\ref{thm:vrt} shows that both the CQMC method and the CQMC+GPCA method have a mean error of $O(n^{-1+\epsilon})$. The purpose of using GPCA is to reduce the effective dimension of $P_1 f$. We thus expect that the CQMC+GPCA method can further enhance the efficiency of the plain CQMC method. One can, of course, use other dimension reduction methods instead of GPCA. Here we only focus on the GPCA method because Xiao and Wang \cite{xiao:wang:2017} found numerically that the GPCA performs consistently better than some common  dimension reduction methods in the literature, such as the LT method.

We carry out numerical experiments using
MATLAB R2013a on a 2.6 GHz CPU. The RQMC points are the linear scrambled Sobol' points proposed by \cite{mato:1998}.
We set parameters in our experiments to  $S(0)=100$, $K=100$, $ \mu=0.01$, $\sigma= 0.4$, $T=1$, and $d\in \{4,20,50\}$. 
The mean errors are estimated based on $200$ repetitions\footnote{Estimation of the mean errors requires knowing the true value of the quantity  being estimated. Here we use the CQMC+GPCA method with a very large sample size  to obtain an accurate estimate of $\E{f(\bm x)}$ and treat it as the true value.}. Figure~\ref{figure:rates} shows the mean errors of the plain MC, the plain QMC, the CQMC, and the CQMC+GPCA methods for the sample sizes $n=2^8,\dots,2^{18}$. When $d=4$, the two CQMC methods (i.e., CQMC and CQMC+GPCA) improve the mean error rate to close to the best possible rate $n^{-1}$. Their mean error rates deteriorate as the dimension $d$ goes up. This is because the mean error of the CQMC methods depends on the dimension $d$, as discussed in Remark~\ref{rem:rate}. Combining the GPCA method with the CQMC method reduces the error satisfactorily. The error rate of the combined method (CQMC+GPCA) declines moderately as the dimension $d$ increases. This highlights the necessity of reducing the effective dimension in QMC. We also observe a similar pattern (not shown here) for the root mean square errors.

\begin{figure}[t]\label{figure:rates}
	\centering
	\includegraphics[width=0.9\hsize]{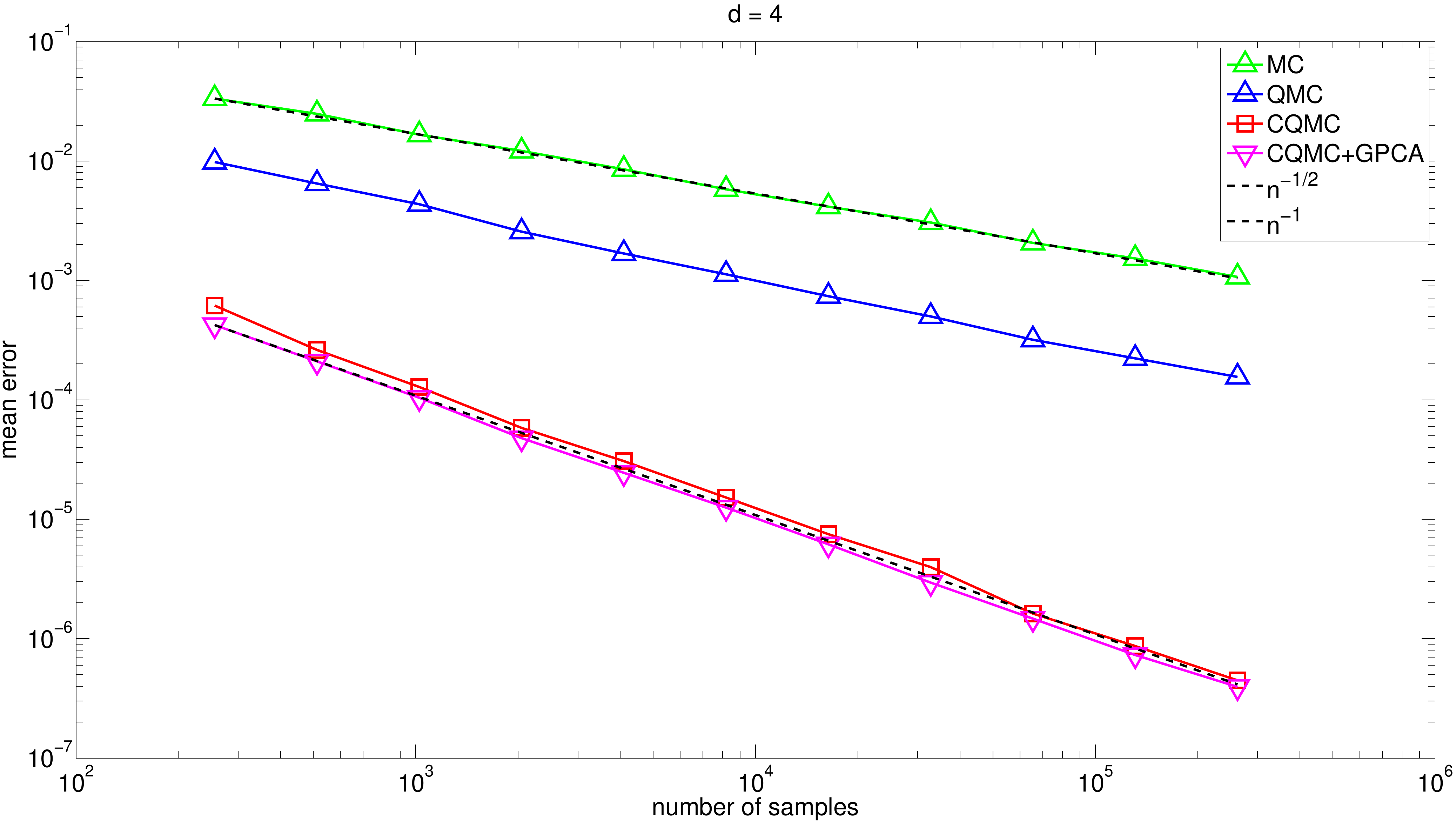}
	\includegraphics[width=0.9\hsize]{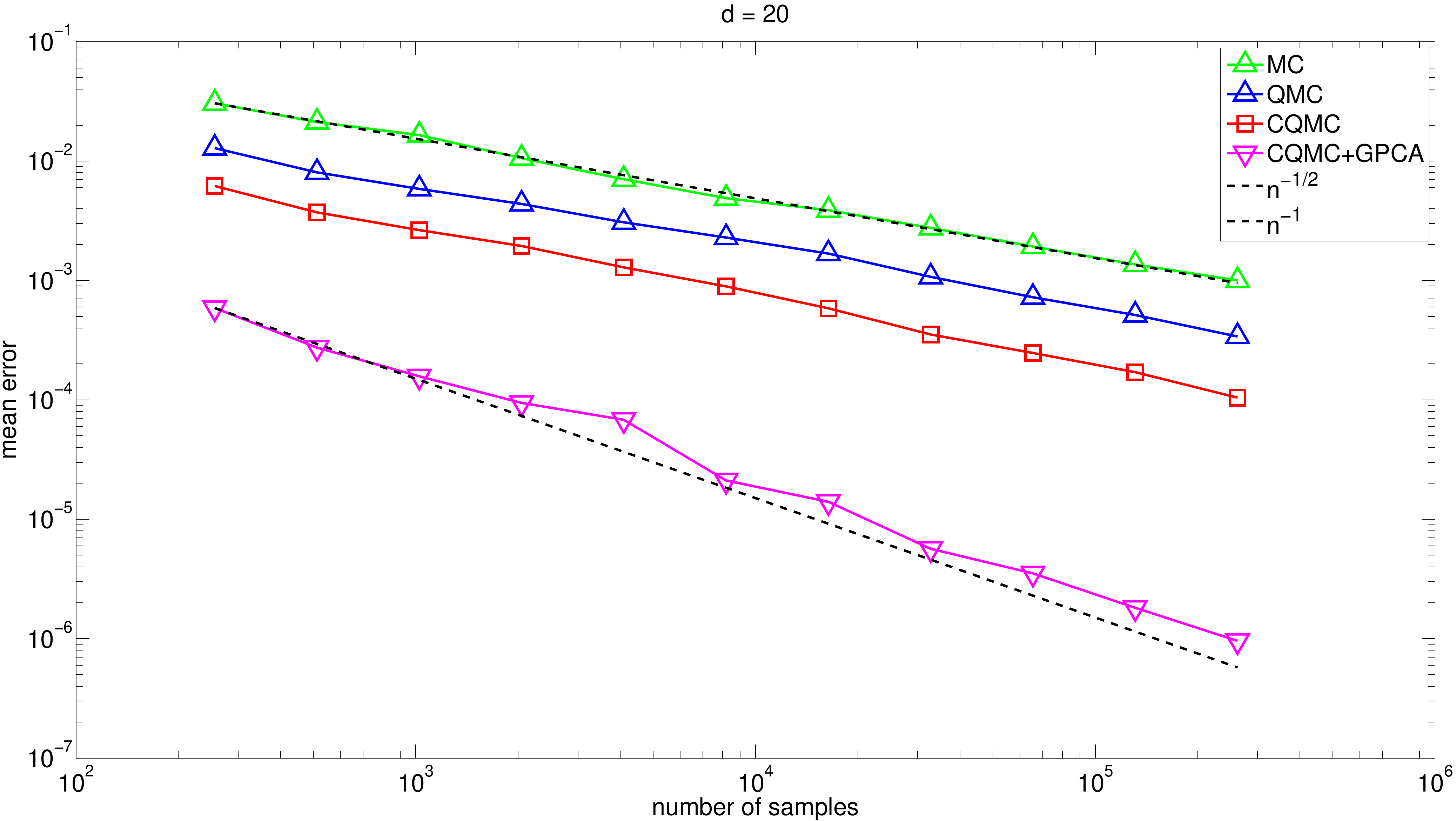}
	\includegraphics[width=0.9\hsize]{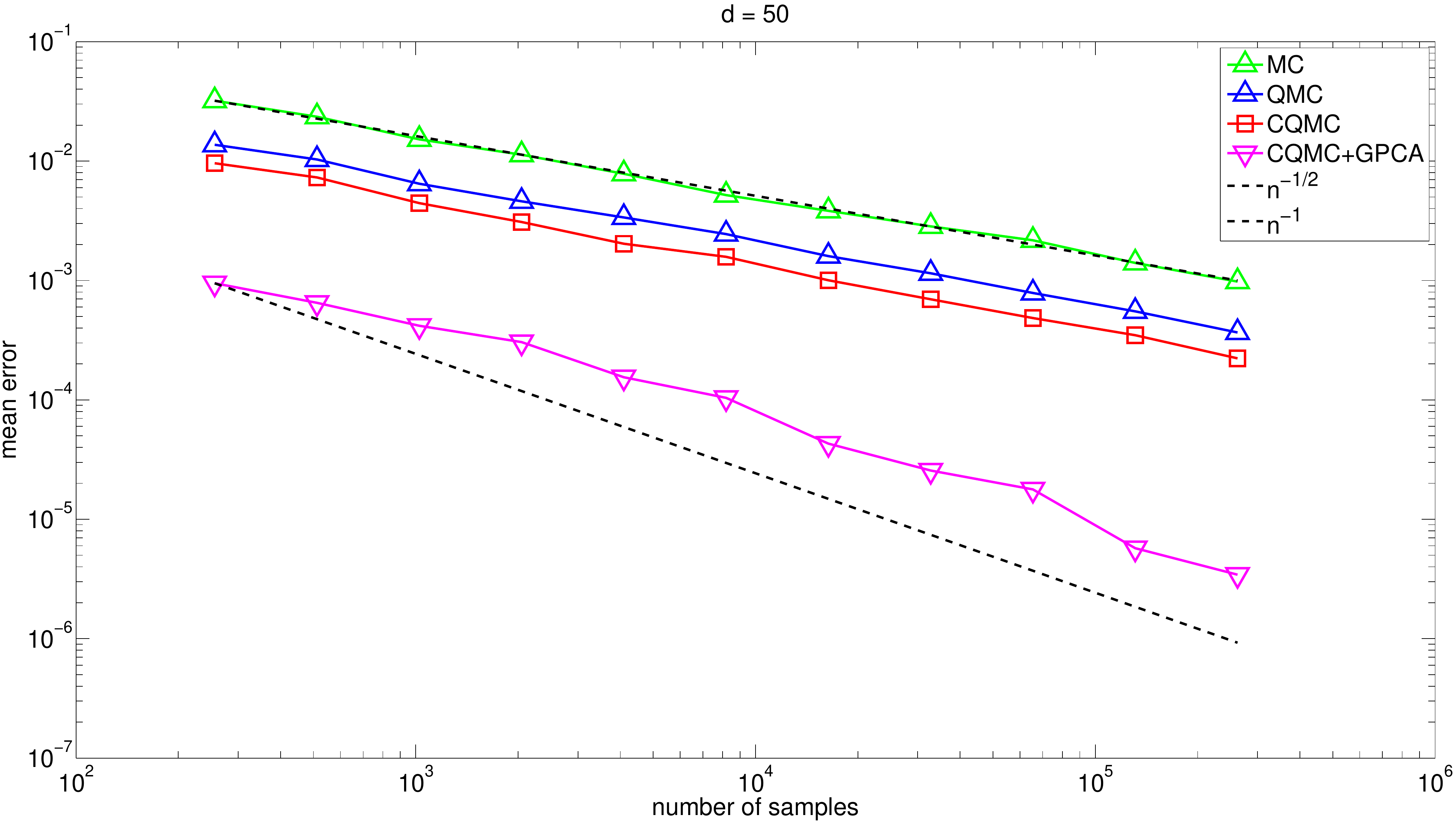}	
	\caption{The mean errors of the plain MC, the plain QMC, the CQMC, and the CQMC+GPCA methods  for Example~\ref{example2} with $d=4,20,50$. The figure has two reference lines proportional to labeled powers
		of $n$. All the mean errors are computed based on 200 runs for $n=2^8,\dots,2^{18}$.}
\end{figure}

\section{Conclusion}\label{sec:concl}
In this paper we found convergence rates of CQMC integration with  discontinuous functions, which typically arise in the pricing and hedging of financial derivatives. The theoretical results show that conditioning not only has the smoothing effect, but also can  bring orders of magnitude reduction in integration error compared to plain QMC.
Under the well-known Black-Scholes framework, we showed that conditioning combined with RQMC yields a mean error of $O(n^{-1+\epsilon})$ for pricing and hedging Asian options. 
This rate also applies to RQMC integration with all the ANOVA terms of discontinuous functions, except the one of the highest order. From this point of view, plain QMC (without conditioning) may be still successful for high-dimensional discontinuous functions with low effective dimension, as observed frequently in option pricing problems (see, e.g., \cite{he:wang:2014,wang:sloan:2011,wang:2013}). 

The rate $O(n^{-1+\epsilon})$ established in this paper also apply to the case of using deterministic Halton sequence as input, thanks to Corollary 5.6 of \cite{owen:2006}. It would be interesting to know how generally this rate holds for other branches of models, beyond the Black-Scholes model.

\section*{Acknowledgments}
	The author thanks Professor Ian Sloan for sharing his work \cite{grie:2017} and his slides presented in MCQMC 2016 conference. The author also thanks Chaojun Zhang for the helpful discussion.

\bibliographystyle{siamplain}
\bibliography{myBiBLibray}

\end{document}